\documentclass{article}
\pdfoutput=1
\usepackage{amsmath} 
\usepackage{amsthm}
\usepackage{amssymb}
\usepackage{latexsym}
\usepackage{tikz-cd}
\usepackage{enumitem}
\usepackage{verbatim}
\usepackage{txfonts}
\usepackage[backend=biber,bibencoding=utf8]{biblatex}
\usepackage{graphicx}
\usepackage{color}
\usepackage{subcaption}

\newtheorem{theorem}{Theorem}[section]
\newtheorem{lemma}[theorem]{Lemma}

\newtheorem{corollary}[theorem]{Corollary}

\newtheorem{conjecture}[theorem]{Conjecture}
\newtheorem{definition}[theorem]{Definition}
\newtheorem{remark}[theorem]{Remark}
\newtheorem{question}[theorem]{Question}

\def\Mod{\mathrm{Mod}}
\def\Modpm{\mathrm{Mod}^{\pm}}
\def\Modpmfin{\mathrm{Mod}^{\pm}_{\mathrm{Fin}}}
\def\Aut{\mathrm{Aut}}
\def\B{\mathrm{B}}

\title{Spaces of Pants Decompositions for Surfaces of Infinite Type}
\author{Ben Branman}

\begin{document}
\maketitle
\begin{abstract}
We study the pants complex of surfaces of infinite type.  When $S$ is a surface of infinite type, the usual definition of the pants graph $\mathcal{P}(S)$ yields a graph with infinitely many connected-components.  In the first part of our paper, we study this disconnected graph.  In particular, we show that the extended mapping class group $\mathrm{Mod}(S)$ is isomorphic to a proper subgroup of $\mathrm{Aut}(\mathcal{P})$, in contrast to the finite-type case where $\mathrm{Mod}(S)\cong \mathrm{Aut}(\mathcal{P}(S))$.  
\par 
In the second part of the paper, motivated by the Metaconjecture of Ivanov \cite{IvanovMeta}, we seek to endow $\mathcal{P}(S)$ with additional structure. To this end, we define a coarser topology on $\mathcal{P}(S)$ than the topology inherited from the graph structure.  We show that our new space is path-connected, and that its automorphism group is isomorphic to $\mathrm{Mod}(S)$.
\end{abstract}
\section{Introduction}
Let $S$ be an orientable surface of infinite type.  The \emph{graph of pants decompositions} $\mathcal{P}(S)$ contains one vertex for each isotopy class of decompositions of $S$ into pairs of pants.  Vertices of $\mathcal{P}(S)$ are connected by an edge if they differ by an ``elementary move'' (see section 2.)
\par 
Unlike in the finite-type case, the pants graph of an infinite-type surface is always disconnected.  To see why, consider a pants decomposition $X$ of $S$ containing only nonseperating curves, and a pants decomposition $Y$ of $S$ containing infinitely many seperating curves.  An elementary move replaces only one curve, hence no finite sequence of elementary moves can take $X$ to $Y$.  
\par 
The (extended) mapping class group $\Mod(S)$ acts on $\mathcal{P}(S)$ by graph isomorphisms.  This action is well understood for finite-type surfaces.  In particular, this action is faithful and $\Aut(\mathcal{P}(S))\cong\Mod(S)$ for all but finitely many surfaces\cite{margalit2002}.  
\par 
The author is aware of only two prior papers which consider pants graphs for infinite-type surfaces.  Fossas and Parlier \cite{FossasParlier} add additional edges to the pants graph to make it connected.  They are concerned primarily with the coarse geometry of the graphs, and do not consider the automorphism group.  Later, Aroca showed \cite{aroca2018remarks} that some of the connected pants graphs of Fossas and Parlier have automorphism group $\Modpm(S)$.

\par 
The outline of the rest of this paper is as follows.  In section 2, we recall some background information.  In section 3, we study the automorphism group of the pants graph of infinite-type surfaces. In particular we show that it is not isomorphic to the mapping class group.  In section 4, we define a family of topologies on the set of isotopy classes of pants decompositions of a surface (that is, the 0-skeleton of the pants graph).  This topology is analogous to the Gromov-Hausdorff topology on a set of metric spaces.  In section 5, we explicitly construct a metric on the vertex space, prove it is complete, and use it to prove that the vertex space is homeomorphic to the Baire space $\mathbb{N}^{\mathbb{N}}$.  
\par 
In section 6, we extend the topology on the vertices to the entire pants graph, creating the \emph{pants space} $\mathcal{PS}(S)$.  In section 7, we prove that the automorphism group of the pants space is the extended mapping class group, generalizing Margalit's result for finite-type surfaces.  In section 8, we prove that the natural action of the mapping class group on $\mathcal{PS}(S)$ is continuous.  
\par 
In section 9, we briefly discuss some open questions about the pants space.
\section{Preliminaries}

In this section we recall the terminology and basic results that we will use in this paper.  
\par 
If $S$ is an orientable surface, a \emph{pair of pants decomposition} $X$ for $S$ is a set of disjoint, simple-closed curves in $S$ such that $S\backslash X$ is homeomorphic to a disjoint union of thrice-punctured spheres.  By abuse of notation, we will sometimes refer to an isotopy class of pants decompositions as a single pants decomposition.  
\par 
The \emph{complexity} $\kappa(S)$ of a surface $S$ is the number of curves in a pants decomposition for $S$.  If $S=S^b_{g,n}$ has finite type, then $\kappa(S)=3g-3+b+n$, otherwise $\kappa(S)=\infty$.  
\par 
We recall the \emph{Alexander Method} for finite-type surfaces \cite[Prop~2.8]{primer}.
\begin{theorem}
Let $S$ have finite-type and let $\kappa(S)\geq 6$, and let $f:S\to S$ be a homeomorphism.  Then $f$ fixes each simple closed curve in $S$ up to isotopy if and only if $f$ is isotopic to the identity on $S$.
\end{theorem}

The main result of \cite{HMValexander} extends the Alexander Method to infinite-type surfaces.
\begin{theorem}[\cite{HMValexander}]
Let $S$ be a surface of infinite type, and let $f:S\to S$ be a homeomorphism.  Then $f$ fixes each simple closed curve in $S$ up to isotopy if and only if $f$ is isotopic to the identity on $S$.
\end{theorem}

Let $S$ be an infinite-type surface and let $S_0\subset S_1\cdots$ be an exhaustion of $S$ by finite-type surfaces.  For notational convenience, we let $S_{-1}=\emptyset$.  Following the convention of \cite{HMValexander}, we say that the $S_j$ form a \emph{principal exhaustion} for $S$ if
\begin{enumerate}
\item
For all $n\geq 0$, $\overline{S_n}\subseteq S_{n+1}^{\mathrm{o}}$.  (This condition ensures that any compact set in $S$ is contained in $S_n$ for some $n$.)
\item
For all $n\geq 0$, $\partial S_n\backslash\partial S$ is a union of finitely many pairwise-disjoint essential simple closed curves.
\item
For all $n\geq 0$, each connected-component of $S_n\backslash S_{n-1}$ has complexity at least 6.
\end{enumerate}

\par

The \emph{curve complex} $C(S)$ of a surface whose vertex set is the set of isotopy classes of simple closed curves on $S$, and where $k$ vertices form a $k-1$-simplex if and only if the corresponding curves have mutually disjoint representatives.  
\par 
The mapping class group acts on $C(S)$ in the obvious way.  The Alexander Method implies that this action is faithful. 1997, Ivanov proved the following:
\begin{theorem}[\cite{Ivanov1997}]
Let $S$ be a finite-type surface of genus at least 2.  Then $\Aut(C(S))\cong \Modpm(S)$.
\end{theorem}

In \cite{Luo1999}, F Luo extended the result to finite-type surfaces of complexity at least 2.  Two decades later, Bavard, Dowdell, and Rafi \cite{BDR} extended the result to infinite-type surfaces.
\par 
There are many related results which state that some simplicial complex associated with a surface has automorphism group $\Modpm(S)$ (see the references in \cite{BM2017} for examples).  In 2006, Ivanov \cite{IvanovMeta} proposed the \emph{Metaconjecture}
\begin{conjecture}[\cite{IvanovMeta}]
Every object naturally associated to a surface $S$ and having sufficiently rich structure has $\Modpm(S)$ as its automorphism group.  Moreover, this can be proven using a reduction to the theorem about the automorphisms of $C(S)$.
\end{conjecture}

The phrases "nautrally associated" and "sufficiently rich structure" are not given rigorous definitions.

\par 
One case of the metaconjecture which was shown by Margalit in \cite{margalit2002} was that of the \emph{graph of pants decompositions}.  Two pants decompositions $X$ and $Y$ are said to \emph{differ by an elementary move} if 
\begin{enumerate}
\item
$X\backslash Y$ is a single curve $\alpha$.
\item
$Y\backslash X$ is a single curve $\alpha'$
\item
The intersection number $\mathrm{i}(\alpha,\alpha')$ is positive minimal for the complexity 1 subsurface containing $\alpha$ and $\alpha'$.
\end{enumerate}

The curves $\alpha, \alpha'$ must lie in a copy of either $S_{1,2}$ or $S_{0,4}$.  Any two nonisotopic curves in $S_{1,2}$ intersect at least once, so for an elementary move we require that they intersect exactly once.  Similarly, any two nonisotopic curves in $S_{0,4}$ intersect at least twice, so for an elementary move we require that they intersect exactly twice.
\par 
The pants graph $\mathcal{P}(S)$ is the graph whose vertices are isotopy classes of pants decompositions of $S$, and where two pants decompositions are joined by an edge if they differ by an elementary move.  There is also a related notion of the \emph{pants complex}, which is a 2-complex whose 1-skeleton is $\mathcal{P}(S)$, but we will not use the pants complex in this paper.

Margalit proved the following theorem.

\begin{theorem}[\cite{margalit2002}]
Let $S$ be a finite-type surface with $\kappa(S)\geq 2$.  Then $\Aut(\mathcal{P}(S))\cong \Modpm(S)$.
\end{theorem}

If $S$ is an infinite-type surface, then $\mathcal{P}(S)$ has uncountably many connected components.  The reason is that an elementary move can only replace one curve at a time.  In an infinite-type surface, each pants decomposition contains infinitely many curves.  For two pair of pants decompositions $X$ and $Y$ such that $X\backslash Y$ is infinite, no finite sequence of elementary moves can turn $X$ into $Y$.  Figure \ref{twoladder} shows an example of two pants decompositions of the ladder surface in different connected-components.

\begin{figure}
\caption{Two pants decompositions of the ladder surface in different components of the pants graph.  Any pants decomposition in the component of (a) must have infinitely many seperating curves.\label{twoladder}}
\def\svgwidth{\columnwidth}

\begin{subfigure}[b]{0.5\textwidth}
\subcaptionbox{\label{laddera}}
{\includegraphics[scale=0.33]{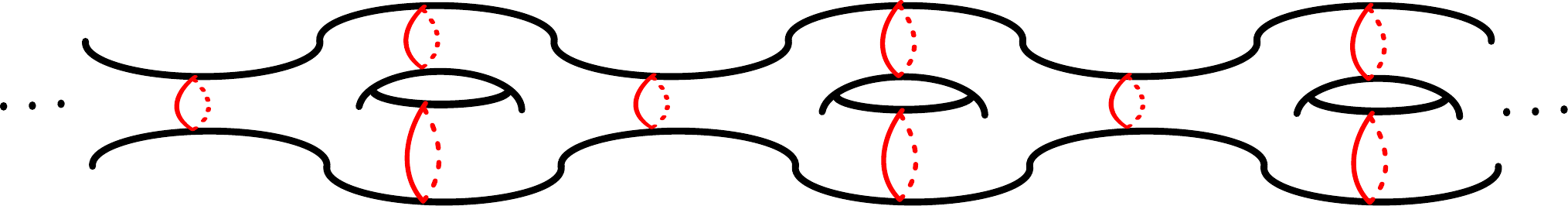}}
\end{subfigure}

\begin{subfigure}[b]{0.5\textwidth}
\subcaptionbox{\label{ladderb}}
{\includegraphics[scale=0.33]{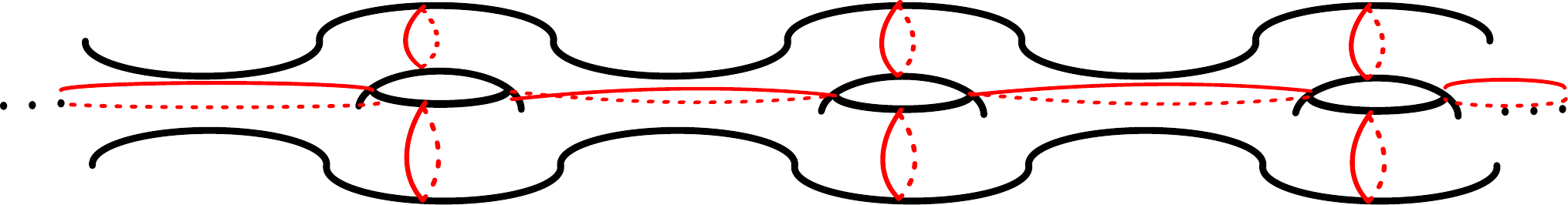}}
\end{subfigure}
\end{figure}

\par 
For any surface $S$, $\Mod(S)$ can be given the structure of a topological group.  We endow $\mathrm{Homeo}(S)$ with the compact-open topology, and quotient by isotopy to obtain a topology on $\Mod(S)$.  When $S$ has finite type, the resulting topology is always discrete \cite[Section 2.1]{primer}.  By contrast, when $S$ has infinite type, $\Mod(S)$ is always homeomorphic to the Baire space $\mathbb{N}^{\mathbb{N}}$ \cite[Corollary 9]{vlamisnotes}.  
 
\section{Automorphisms of $\mathcal{P}$}
The goal of this section is to understand the automorphism group of $\mathcal{P}(S)$.

\subsection{}
The main goal of this subsection is to establish the following result.

\begin{theorem}
Let $S$ be an infinite-type surface.  Then the extended mapping class group of $S$ is isomorphic to a proper subgroup of $\Aut(\mathcal{P}(S))$.
\end{theorem}

This theorem at first glance seems to contradict the Ivanov Metaconjecture.  However, $\mathcal{P}(S)$ is a disconnected graph with uncountably many connected components.  Thus, it can be said to lack the "sufficiently rich structure" required by the Metaconjecture.
\par 
We will need several lemmas to prove Theorem 3.1.

\begin{lemma}
Let $S$ be an infinite-type surface, and let $X$ and $Y$ be pants decompositions of $S$.  Then $X$ and $Y$ lie in the same connected-component of $\mathcal{P}(S)$ if and only if $|X\backslash Y|$ is finite.
\end{lemma}

\begin{proof}
The "only if" direction follows from the fact that an elementary move can only replace one curve at a time.  Conversely, if $|X\backslash Y|$ is finite, then $|Y\backslash X|=|X\backslash Y|$ is also finite.  Hence the union of all curves in $(X\backslash Y)\cup (Y\backslash X)$ is contained in some finite-type subsurface of $S$.  The "if" direction now follows from the fact that finite-type pants graphs are connected.
\end{proof}

\begin{lemma}
Let $S$ be a surface of infinite type and let the support of $f\in \Modpm(S)$ have finite type.  Then $f$ induces automorphisms on every connected component of $\mathcal{P}(S)$. 
\end{lemma}
\begin{proof}
If $f$ has finite type support and $X$ is a pants decomposition of $S$, then $f$ fixes all but finitely-many curves of $X$ up to isotopy.  Hence, $f(X)$ and $X$ are in the same connected-component of $\mathcal{P}(S)$.
\end{proof}

\begin{lemma}
Let $\Gamma$ be a connected-component of $\mathcal{P}(S)$ and let $f\in \Modpm(S)$ be a mapping class such that the restriction of $f$ to $\Gamma$ is the identity.  Then $f$ is isotopic to the identity.
\end{lemma}

\begin{proof}
Let $a$ and $b$ be distict curves in $S$.  Then there is some pants decomposition $X\in \Gamma$ such that $a\in X$ but $b\notin X$.  By construction, $f(X)=X$, hence $f(a)\neq b$.
Since the only assumption we needed to make about $b$ was that $b\neq a$, we can conclude that $f(a)=a$.  Hence, by the Alexander method, $f\cong \text{Id}_S$
\end{proof}

\begin{lemma}
Let $\Gamma$ be a connected-component of $\mathcal{P}(S)$ and let $f,g\in \Modpm(S)$ be mapping classes such that $f_{|_{\Gamma}}=g_{|_{\Gamma}}$.  Then $f$ and $g$ are isotopic.
\end{lemma}

\begin{proof}
Since the restrictions of $f$ and $g$ to $\Gamma$ are equal, the restriction of $f^{-1}\circ g$ to $\Gamma$ is the identity.  Hence $f^{-1}\circ g$ is isotopic to the identity by the previous lemma.
\end{proof}
Let $\Modpmfin(S)$ denote the group of mapping classes which contain a homeomorphism whose support is a finite-type subsurface.  
\begin{lemma}
The automorphism group $\Aut(\mathcal{P})$ contains a subgroup isomorphic to the direct product
\[
\prod_{\Gamma\in \Pi_0(\mathcal{P}(S))} \Modpmfin(S)
\]
\end{lemma}

\begin{proof}
$\mathcal{P}$ is a disconnected graph, and hence its automorphism group contains a copy of
\[
\prod_{\Gamma\in \Pi_0(\mathcal{P}(S))} \Aut(\Gamma)
\]
The finite-type-supported mapping class group acts faithfully on each connected component of $S$, and hence $\Aut(\Gamma)$ contains a subgroup isomorphic to $\Modpmfin(S)$.  The result follows.
\end{proof}

We are now ready to prove Theorem 3.1.

\begin{proof}
Choose some component $\Gamma\in \Pi_0(\mathcal{P}(S))$ and some nontrivial mapping class $f\in \Modpmfin(S)$.  Then $f$ induces a nontrivial automorphism on $\Gamma$.  Let $\overline{f}$ be the automorphism of $\mathcal{P}(S)$ given by
\[
\overline{f}(X)=
\begin{cases}
f(X), & X\in \Gamma\\
X, & X\notin \Gamma
\end{cases}
\]
Then $\overline{f}$ gives an automorphism on each component of $\mathcal{P}(S)$ and hence is an automorphism of $\mathcal{P}(S)$.  On the other hand, $\overline{f}$ acts trivially on all but one component, and hence by Lemma 3.4 cannot be induced by a nontrivial mapping class.
\end{proof}

\subsection{}

The main goal of this section is to prove a converse to Lemma 3.3.  Namely, we will show

\begin{theorem}
Let $f\in \Modpm(S)$.  Then $f$ induces automorphisms on every connected component of $\mathcal{P}(S)$ if and only if $f$ is finite-type-supported.
\end{theorem}

We will need several lemmas to prove this theorem.

\begin{lemma}
Let $S$ be an infinite-type surface, $X$ be a pants decomposition of $S$, and $C$ be an infinite subset of the curves of $X$.  Then there exists an infinite subset $A=\{\alpha_1,\ldots,\}\subseteq C$ such that all curves in $A$ lie on pairwise-disjoint complexity 1 subsurfaces.
\end{lemma}

\begin{proof}
We will construct $A$ inductively.  We can choose any curve in $C$ to be $\alpha_1$.  Now suppose we have already chosen $\alpha_1,\ldots, \alpha_{n-1}\in C$ which all lie on disjoint complexity 1 subsurfaces.  Each $\alpha_i$ shares a pair of pants with at most four curves in $C\backslash \{\alpha_1,\ldots, \alpha_{n-1}\}$.  Hence, all but finitely many curves in $C$ lie on disjoint complexity 1 subsurfaces from each $\alpha_i$ for $1\leq i\leq n-1$, and hence there are infintely many choices in $C$ for $\alpha_n$.
\end{proof}

\begin{lemma}
Let $X$ be a pants decomposition of $S$ and let $A=\{\alpha_1,\ldots,\}\subseteq X$ be an infinite family of curves in $X$ which all lie on mutually disjoint complexity 1 subsurfaces.  Let $f\in \Modpm(S)$ be a mapping class such that $f(\alpha_i)\neq \alpha_i$ for all $i\in \mathbb{N}$.  Then there exists a family of curves $B=\{\beta_1,\ldots,\ldots\}$ in $S$ such that
\begin{enumerate}
\item
The curve $\beta_i$ is disjoint from $X\backslash\{\alpha_i\}$.
\item
The curves $\alpha_i$ and $\beta_i$ are distict.
\item
For all $i\in \mathbb{N}$, $f(\beta_i)\notin B$

\item
For all $i\in \mathbb{N}$, $f^{-1}(\beta_i)\notin B$.
\end{enumerate}
\end{lemma}

\begin{proof}
We can construct $B$ inductively.  As a base case, $B_0=\emptyset$ vacuously satisfies conditions (1)-(4).  As our induction hypothesis, suppose that we have chosen a sequence of curves $B_{n-1}=\{\beta_1,\ldots, \beta_{n-1}\}$ such that for each $1\leq i\leq n-1$, $\beta_i$ satisfies the conditions
\begin{enumerate}
\item
The curve $\beta_i$ is disjoint from $X\backslash\{\alpha_i\}$.
\item
The curves $\alpha_i$ and $\beta_i$ are distict.
\item
For all $1\leq i\leq n-1$, $f(\beta_i)\notin B_{n-1}$

\item
For all $1\leq i\leq n-1$, $f^{-1}(\beta_i)\notin B_{n-1}$.
\end{enumerate}

We want to choose a curve $\beta_n$ such that
\begin{enumerate}[label=(\roman*)]
\item
The curve $\beta_n$ is disjoint from $X\backslash\{\alpha_n\}$.
\item
The curves $\alpha_n$ and $\beta_n$ are distinct

\item
$f(\beta_n)\notin B_{n-1}$

\item
$f^{-1}(\beta_n)\notin B_{n-1}$

\item
$f(\beta_n)\neq \beta_n$
\end{enumerate}
Conditions (i)-(iv) can be easily satisfied: there are infinitely many isotopy classes of curves in the complexity 1 subsurface containing $\alpha_n$, and conditions (ii)-(iv) only require us to avoid a finite number of curves.  Now recall that $f(\alpha_n)\neq \alpha_n$.  Hence, $f$ is not isotopic to the identity on the complexity 1 subsurface containing $\alpha_n$, and hence there are infinitely many curves on this subsurface which are not fixed by $f$.  Thus, we can find a $\beta_n$ satisfying conditions (i)-(v).  
\par
Conditions (iii) and (v) ensure that $f^{-1}(\beta_i)\neq \beta_n$ for $1\leq i\leq n$. Conditions (iv) and (v) ensure that that $f(\beta_i)\neq \beta_n$ for $1\leq i\leq n$.  We have thus shown that $B_n:=B_{n-1}\cup \{\beta_n\}$ satisfies conditions (1)-(4) of the induction hypothesis.  Hence, the desired set of curves $B$ exists by induction.
\end{proof}

\begin{lemma}
Let $f\in \Modpm(S)$ and let $X$ be a pants decomposition of $S$.  Suppose that there are infinitely many curves $c\in X$ such that $f(c)\nsimeq c$  Then there exists a pants decomposition $Y$ such that $Y$ and $f(Y)$ are in different components of $\mathcal{P}(S)$.
\end{lemma}
\begin{proof}
If $X$ and $f(X)$ are in different components, then we are done.
\par 
Now suppose that $X$ and $f(X)$ are in the same connected component.  Our goal is to modify $X$ in such a way that we will obtain the desired pants decomposition $Y$.
\par 
Because $X$ and $f(X)$ are in the same component, the set of curves $X\backslash f^{-1}(X)$ is finite.  Hence, there are infinitely many curves $c\in X$ such that $f(c)\neq c$ but $f(c)\in X$.  Thus, we can find an infinite subset $A=\{\alpha_1,\ldots\}\subseteq X$ such that $f(A)\subseteq X$, $f(\alpha_i)\neq (\alpha_i)$ for all $i\geq 1$, and all the curves in $A$ lie on mutually disjoint subsurfaces.  
\par 

Invoking the previous lemma, we obtain a family of curves $B=\{\beta_1,\ldots\}$.  We obtain $Y$ from $X$ by replacing $\alpha_i$ with $\beta_i$ for all $i\geq 1$.  This process still yields a pants decomposition because the $\alpha_i$ all lie on disjoint subsurfaces.    
\par
Now, we want to show that $f(\beta_i)\notin Y$ for all $i\geq 1$.  First, we consider the case where $f(\alpha_i)\notin A$.  Then $f(\alpha_i)\in Y$.  Also, $\alpha_i$ and $\beta_i$ are distinct and intersect, so $f(\alpha_i)$ and $f(\beta_i)$ are also distinct and intersect.  Thus, $f(\beta_i)$ cannot be contained in a pants decomposition with $f(\alpha_i)$, so $f(\beta_i)\notin Y$.

\par 
On the other hand, suppose $f(\alpha_i)\in A$.  Then $f(\alpha_i)=\alpha_j$ for some $j\in \mathbb{N}$, and $\beta_j\in Y$.  Since $\alpha_i$ and $\beta_i$ lie on the same complexity 1 subsurface, so do $\alpha_j$, $\beta_j$, and $f(\beta_i)$.  Recall that $\beta_j$ was chosen so that $f(\beta_i)\neq \beta_j$, hence $f(\beta_i)\notin Y$.  
\par
Thus, we have infinitely many curves that are in $f(Y)$ but not $Y$, so $f(Y)$ and $Y$ are in different components.

\end{proof}

We are now ready to prove Theorem 3.7
\begin{proof}
We have already proven the ``if" direction in Lemma 3.3.  Now, suppose that $f$ induces automorphisms on every connected component of $\mathcal{P}(S)$.  Choose a finite set of pants decompositions $P_1, \ldots, P_n$ which fill $S$.  By the previous lemma, $f$ fixes all but finitely many curves in each pants decomposition.  The mapping class group of a pair of pants is generated by Dehn twists around the boundary components.  Hence, outside of a finite type subsurface, $f$ is isotopic to a product of Dehn twists along curves in $P_1$.  But the $P_i$ fill $S$, hence each curve in $P_1$ intersects with at least one curve in $P_j$ for some $2\leq j\leq n$.  Since $f$ also fixes all but finitely many curves in $P_j$ for all $2\leq j\leq n$, it follows that $f$ is isotopic to the identity on all but finitely many pairs of pants in $P_1$, and hence $f$ is finite-type-supported.
\end{proof}

\subsection{}
The goal of this subsection is to prove the following two theorems.  Let $\phi:S\to S'$ be a $\pi_1$-injective simplicial embedding between surfaces.  Let $\Gamma$ and $\Gamma'$ respectively be connected components of $S$ and $S'$, and let $f:\Gamma\to \Gamma'$ be a simplicial injective embedding.  We say that $\phi$ \emph{induces} $f$ if there exists a multicurve $Q$ of $S'$ such that $f(X)=\phi(X)\cup Q$ for all vertices $X$ of $\Gamma$.
\begin{theorem}
Let $\Gamma$ and $\Gamma'$ be connected components of $\mathcal{P}(S)$ and $\mathcal{P}(S')$, respectively.  let $f: \Gamma\to \Gamma'$ be an isomorphism.  Then there is a homeomorphism $\phi: S\to S'$, unique up to isotopy, which induces $f$.
\end{theorem}
This result will be important in section seven when we determine the automorphism group of the pants space.   

We will need the following theorem, which is an extension of the main result of \cite{Aram2010}.

\begin{theorem}[Pants Graph Embedding Theorem]
Let $S$ and $S'$ be (possibly finite-type) surfaces such that each connected-component has complexity at least 2.  Let $\Gamma$ and $\Gamma'$ respectively be connected components of $\mathcal{P}(S)$ and $\mathcal{P}(S')$.  Let $f: \Gamma\to \Gamma'$ be an injective simplicial map.  Then there exists a $\pi_1$-injective embedding $\phi: S\to S'$ which induces $f$.
\end{theorem}
The case where both $S$ and $S'$ are of finite-type is Theorem A in \cite{Aram2010}.  
\par

We start by proving the case of the Pants Graph Embedding Theorem where $S$ has finite-type and $S'$ has infinite type.  

\begin{lemma}
Let $S$ be a compact orientable surface of negative Euler characteristic and such that each connected-component has complexity at least 2.  Let $S'$ be an infinite-type orientable surface, and let $f: \mathcal{P}(S)\to \mathcal{P}(S')$ be an injective simplicial map.  Then there exists a $\pi_1$-injective embedding $\phi: S\to S'$ which induces $f$.
\end{lemma}

\begin{definition}\rm
Let $Q$ be a multicurve of $S$.  The \emph{deficiency} of $Q$ is the cardinality of $P\backslash Q$, where $P$ is a pants decomposition containing $Q$.  Note that if $S$ is of finite type, then the deficiency of $Q$ is $\kappa(S)-|Q|$.
\end{definition}
We note that in the case where $S$ has finite type, our definition is equivalent to the one given in \cite{Aram2010}.
\par 
With our altered definition of deficiency, the proof Lemma 3.13 is exactly the same as the case where both surfaces are of finite-type.  Writing out the full proof would simply involve copying Aramayona's paper word for word, except that our surface $S'$ has infinite type.   

\par 
In particular, Aramayona's proof gives us a lemma similar to Theorem C in \cite{Aram2010}.  Let $Q\subset S$ be a multicurve.  We denote by $\mathcal{P}(S)_Q$ the full subgraph of $\mathcal{P}(S)$ spanned by vertices which contain $Q$.  Note that if $Q$ has finite deficiency, then $\mathcal{P}(S)_Q$ is connected by Lemma 3.2.  A connected-component of $S\backslash Q$ is said to be \emph{nontrivial} if it has nonzero complexity.    

\begin{lemma}
Let $S$ be a finite-type surface such that each connected-component has complexity at least 2, and let $S'$ be an infinite-type surface.  Let $\Gamma'$ be a connected component of $\mathcal{P}(S')$.  Let $f: \mathcal{P}(S)\to \Gamma'$ be an injective simplicial map.   Then
\begin{itemize}
\item
There exists an infinite multicurve $Q\subset S'$, of deficiency $\kappa(S)$, such that $f(\mathcal{P}(S))=\mathcal{P}(S')_Q$.  In particular, $\mathcal{P}(S)\cong \mathcal{P}(S')_Q$.

\item 
$S$ and $S'\backslash Q$ have the same number of nontrivial components.  Moreover, if $S_1,\ldots, S_r$ and $S'_1,\ldots, S'_r$ are respectively  nontrivial components of $S$ and $S'\backslash Q$, then, up to reordering of the indices, $f$ induces isomorphisms $f_i:\mathcal{P}(S_i)\to \mathcal{P}(S'_i)$.  In particular $\kappa(S_i)=\kappa(S'_i)$.  
\end{itemize}
\end{lemma}

We now present a proof of Theorems 3.12.  This argument is similar to the one used in \cite{BDR} to prove Theorem 1.3.

\begin{proof}[Proof of Theorem 3.12]
Choose a pants decomposition $X\in \Gamma$.  Choose an exhaustion $S_0\subset S_1\subset \cdots \subset S$ of $S$ by finite-type surfaces such that for all $n\geq 0$, $\partial S_n$ is a disjoint union of essential curves in $X$.  This exhaustion induces natural embeddings $\iota_n: \mathcal{P}(S_n)\to \Gamma$.  Composing with $f$ gives embeddings $f\circ \iota_n: \mathcal{P}(S_n)\to \Gamma'$.  
\par 
By Lemma 3.15, each $f\circ \iota_n$ is an isomorphism between $\mathcal{P}(S_n)$ and $\mathcal{P}(S')_{Q_n}$ for some deficiency $\kappa(S_n)$ multicurve $Q_n$.  Since $f\circ \iota_n$ is a restriction of $f\circ \iota_{n+1}$, $Q_{n+1}\subset Q_n$ for all natural numbers $n$.   
\par 
By Lemma 3.13, each $f\circ \iota_n$ is induced by a $\pi_1$-injective embedding $\phi_n:S_n\to S'$.  Moreover, by the second part of Lemma 3.15, the image of $\phi_n$ is the union of the nontrivial components of $S'\backslash Q_n$.  Since $Q_{n+1}\subset Q_n$, the image of $\phi_n$ is a subsurface of the image of $\phi_{n+1}$, and moreover $\phi_n$ is isotopic to a restriction of $\phi_{n+1}$.  Hence we can take the colimits of the $\phi_n$ to get a $\pi_1$-injective embedding $\phi:S\to S'$.
\end{proof}

Now we are ready to prove Theorem 3.11.

\begin{proof}[Proof of Theorem 3.11]
Our argument here is almost the same as the proof of Theorem 3.12.  The only difference is that we need to show the embedding $\phi:S\to S'$ is actually a homeomorphism.  Suppose $S\backslash S'$ is nonempty.  Since $\phi$ is $\pi_1$-injective, there is an essential curve $\alpha\subset (S\backslash S')$.  Since $\phi$ induces $f$, $\alpha$ is either a curve in every curve of $f(\Gamma)$ or none of the curves of $f(\Gamma)$.  But this is impossible as $f$ is surjective.  
\end{proof}

\subsection{}
The main goal of this section is to prove the following theorem.
\begin{theorem}
Let $\Gamma\in \Pi_0(\mathcal{P}(S))$ be a connected-component let $X\in \Gamma$ be a pants decomposition.  Then there is an infinite-support mapping class $f\in \overline{\Modpmfin(S)}\backslash \Modpmfin(S)$ such that $f(X)=X$.  
\end{theorem}

\begin{proof}
Choose an orientation for each curve $X$.  Let $f$ be the mapping class of the homeomorphism obtained by performing a Dehn twist simultaneously along each curve in $X$.  Then $f$ fixes $X$ and hence induces an automorphism on $\Gamma$. If we let $\alpha_1,\alpha_2,\ldots$ denote the curves of $X$, then $f$ is a limit of the sequence $T_{\alpha_1},T_{\alpha_1}\cdot T_{\alpha_2}, T_{\alpha_1}\cdot T_{\alpha_2}\cdot T_{\alpha_3},\ldots$ of compactly supported mapping classes. Hence $f$ lies in the closure of $\Modpmfin(S)$.
\end{proof}

\subsection{}
The main goal of this section is to prove the following theorem.
\begin{theorem}
Let $\Gamma\in \Pi_0(\mathcal{P}(S))$ be a connected-component.  Then there is a mapping class $f\in \overline{\Modpmfin(S)}\subset \Modpm(S)$ such that $f(\Gamma)\neq \Gamma$. 
\end{theorem}

First, we recall a theorem of Taylor and Zupan\cite{taylor2014products}.

\begin{theorem}
Let $S$ be a finite surface, let $G$ be a Cartesian product of Farey graphs, and let $i:G\to \mathcal{P}(S)$ be a simplicial embedding.  Then $i(G)$ is totally geodesic in $\mathcal{P}(S)$.
\end{theorem}

To prove Theorem 3.17, we will need and infinite-type version of Taylor and Zupan's theorem.  We will need the following lemma.

\begin{lemma}
Let $S$ have infinite type, let $\Gamma\in \Pi_0(\mathcal{P}(S))$ be a connected-component, let $G$ be a finite graph, and let $\iota:G\to \Gamma$ be a simplicial embedding.  Then there is a finite-type subsurface $S'\subset S$ such that $\iota(G)\subset i(\mathcal{S'})$, where $i:\mathcal{P}(S')\to \mathcal{P}(S)$ is the natural inclusion.
\end{lemma}

\begin{proof}
Let $X_1,\ldots, X_m$ be the vertices of $\iota(G)$.  Since $\Gamma$ is connected, Lemma 3.2 implies that 
\[X_j\backslash (\bigcap_{k=1}^m X_k)\]
is finite for each $1\leq j\leq m$.  Hence, 
\[A=\bigcup_{j=1}^m (X_j\backslash (\bigcap_{k=1}^m X_k))\]
is also finite.  Let $S'$ be a finite-time surface containing every curve in $A$.  Then $S'$ satisfies the conclusion of the lemma.
\end{proof}

\begin{lemma}
Let $S$ have infinite type, let $\Gamma\in \Pi_0(\mathcal{P}(S))$ be a connected-component, let $G$ be a Cartesian product of finitely many Farey graphs, and let $i:G\to \Gamma$ be a simplicial embedding.  Then $i(G)$ is totally geodesic in $\Gamma$.
\end{lemma}

\begin{proof}
Let $X$ and $Y$ be vertices of $G$.  Let $P$ be a geodesic  in $\Gamma$ from $i(X)$ to $i(Y)$.  Then $P\cup i(G)$ is a finite graph, and hence by Lemma 3.19 there is a finite-type subsurface $S'\subset S$ such that $(P\cup i(G))\subset \mathcal{P}(S')$.  By Taylor and Zupan's theorem for finite-type surfaces, $i(G)$ is totally geodesic in $\mathcal{P}(S')$, and hence $P\subset i(G)$, as desired.
\end{proof}

Before proving Theorem 3.17, we note that Lemma 3.20 provides us with an interesting corollary about the pants graph.  Recall that the \emph{geometric rank} of a metric space $G$ is the greatest $n$ such that there is a quasi-isometric embedding of $\mathbb{Z}^n$ in $G$.  

\begin{corollary}
Let $S$ have infinite type, and let $\Gamma\in \Pi_0(\mathcal{P}(S))$ be a connected-component.  Then $\Gamma$ has infinite geometric rank.
\end{corollary}

\begin{proof}
For any $n\in \mathbb{N}$, $S$ contains a subsurface $S'$ consisting of $n$ disjoint complexity 1 subsurfaces.  The pants graph of $S'$ is the product of $n$ Farey graphs.  By Lemma 3.20, this product embeds totally geodesicly in $\Gamma$, hence the geometric rank of $\Gamma$ is at least $n$.  
\end{proof}

Now we are ready to prove Theorem 3.17.

\begin{proof}
Pick a pants decomposition $X\in \Gamma$.  Pick a multicurve $Y\subset X$ such that $S\backslash Y$ is a disjoint union of complexity 1 surfaces.  Call these complexity 1 surfaces $S_1,\ldots$.  Let $X_j=X\cap S_j$ for $j\in \mathbb{N}$.  For each $j\in \mathbb{N}$, choose a pseudo-Anosov map $\phi_j$ on $S_j$ such that $d_{\mathcal{P}(S_j)}(X_j,\phi_j(X_j))>j$.  Let $\widetilde{\phi_j}\in \mathrm{Homeo}(S)$ be an extension of $\phi_j$ to $S$ which is the idenity outside of $\overline{S_j}$. By Lemma 3.20, the image of the natural inclusion of $\mathcal{P}(S_j)$ is totally geodesic in $\Gamma$, hence $d_{\Gamma}(X,\widetilde{\phi_j}(X))>j$.  Consider the homeomorphism $\widetilde{\phi}$ obtained as the limit of the sequence $\widetilde{\phi_1},\widetilde{\phi_2}\circ \widetilde{\phi_1}, \widetilde{\phi_3}\circ \widetilde{\phi_2}\circ \widetilde{\phi_1},\ldots$.  Then $\widetilde{\phi}(X)$ cannot be any finite distance from $X$, and hence $\widetilde{\phi}(X)\notin \Gamma$.  Thus the mapping class of $\widetilde{\phi}$ satisfies the conclusion of Theorem 3.17.
\end{proof}

We now state a corollary which ties together several results from this section.

\begin{corollary}
Let $S$ have infinite type, and let $\Gamma\in \Pi_0(\mathcal{P}(S))$ be a connected-component.  There are natural monomorphisms $\Modpmfin(S) \hookrightarrow \Aut(\Gamma)\hookrightarrow \Modpm(S)$.  Neither map is a surjection.  Moreover, the complement of the image of both maps contain elements of $\overline{\Modpmfin(S)}$.  
\end{corollary}

\begin{proof}
The first embedding comes from Lemma 3.3.  The fact that the first embedding is not surjective, and that the complement of the image contains an element of $\overline{\Modpmfin(S)}$, comes from Theorem 3.16.  
\par 
The second embedding comes from Theorem 3.11.  The fact that the second embedding is not surjective, and that the complement of the image contains an element of $\overline{\Modpmfin(S)}$, comes from Theorem 3.17.
\end{proof}

\section{The Vertex Spaces $V_i(S)$}

In this section, we endow the vertex set of $\mathcal{P}(S)$ with a family of topologies.  We will prove that the resulting vertex spaces are metrizable, second-countable, and totally disconnected.  We will use the vertex spaces again in the next section.
\par

We fix a principal exhaustion $S_0\subseteq S_1\subseteq \cdots$ of $S$ by finite-type subsurfaces.  Also fix a hypberbolic metric on $S$.  For pants decompositions $X$ and $Y$, we wish to define some sense in which $X$ and $Y$ "agree" on a compact subsurface $Z\subset S$.  We wish for this definition to be preserved by isotopy.  We will use this notion of agreement to define the limit of a sequence of pants decompositions, and use this definition of a limit to get a topological structure on the set of pants decompositions.
\par 
For the chosen hyberbolic metric on $S$, each pants decomposition has a unique geodesic representative.  Thus, we can define agreement for the geodesic represenatives, and consider two pants decompositions to agree if their geodesic representatives agree.

\par 
\begin{definition}\rm
Two pants decompositions $X$ and $Y$ \emph{0-agree} if they are isotopic.
\end{definition}

\begin{definition}\rm
Let $X$ and $Y$ be two geodesic pants decompositions, viewed as sets of (geodesic) curves, and let $S'\subset S$ be a compact subsurface.  $X$ and $Y$ are said to \emph{1-agree on $S'$} if 
\[\{x\in X: x\cap S'\neq \emptyset\}=\{y\in Y: y\cap S'\neq \emptyset\}\]
\end{definition}

There is another way we could have defined 1-agreement.

\begin{lemma}
Let $X$ and $Y$ be two geodesic pants decompositions, viewed as subsets of $S$.  Then $X$ and $Y$ 1-agree on $S'$ if and only if
\[X\cap S'=Y\cap S'\]
\end{lemma}
\begin{proof}
The "only if" direction is trivial.  Suppose that $X\cap S'=Y\cap S'$ and that $\gamma$ is a curve in $X$ which intersects $S'$.  Then $Y$ must contain a curve $\gamma'$ such that $\gamma\cap S'=\gamma'\cap S'$.  Any two tangent geodesics are equal, so $\gamma'=\gamma$.  Thus $X$ and $Y$ 1-agree on $S'$.
\end{proof}

\begin{definition}\rm
Let $X$ and $Y$ be two geodesic pants decompositions, and let $S'\subset S$ be a compact subsurface.  View $X$ and $Y$ as subsets of $S$.  $X$ and $Y$ are said to \emph{2-agree on $Z$} if the subsets $X\cap S'$ and $Y\cap S'$ are properly isotopic, counting multiplicities of proper arcs. 
\end{definition}

Figure \ref{2not1} shows a pair of pants decompositions which 2-agree but do not 1-agree.  

\begin{figure}
\def\svgwidth{\columnwidth}
\caption{Two pants decompositions of $S$.  The black circle is $\partial S'$, the red dots are punctures, and the blue curves are pants curves.  The two decompositions 2-agree on $S'$ but do not 1-agree on $S'$.\label{2not1}}

\subcaptionbox{\label{2not1a}}
{\includegraphics[scale=0.25]{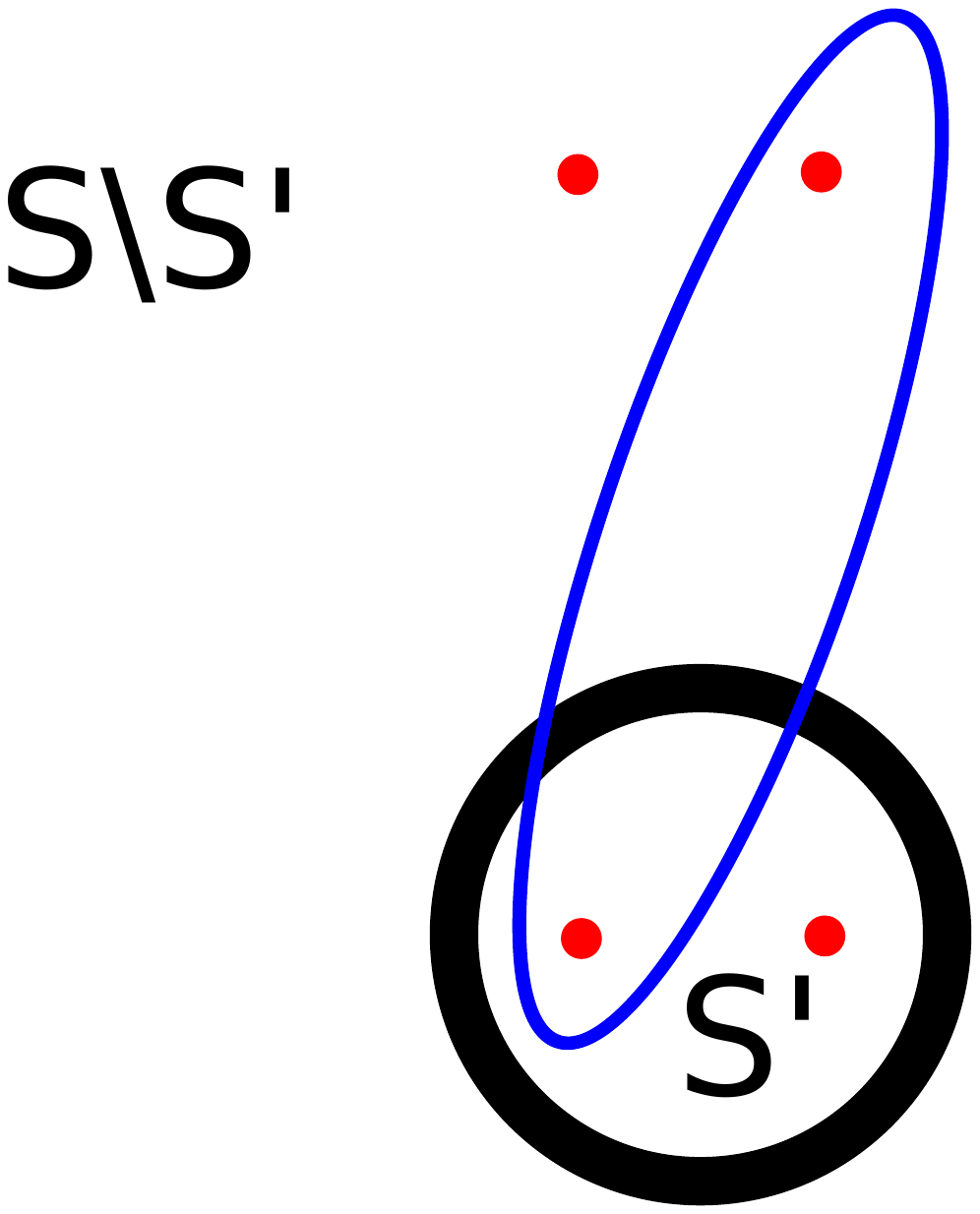}}
\subcaptionbox{\label{2not1a}}
{\includegraphics[scale=0.25]{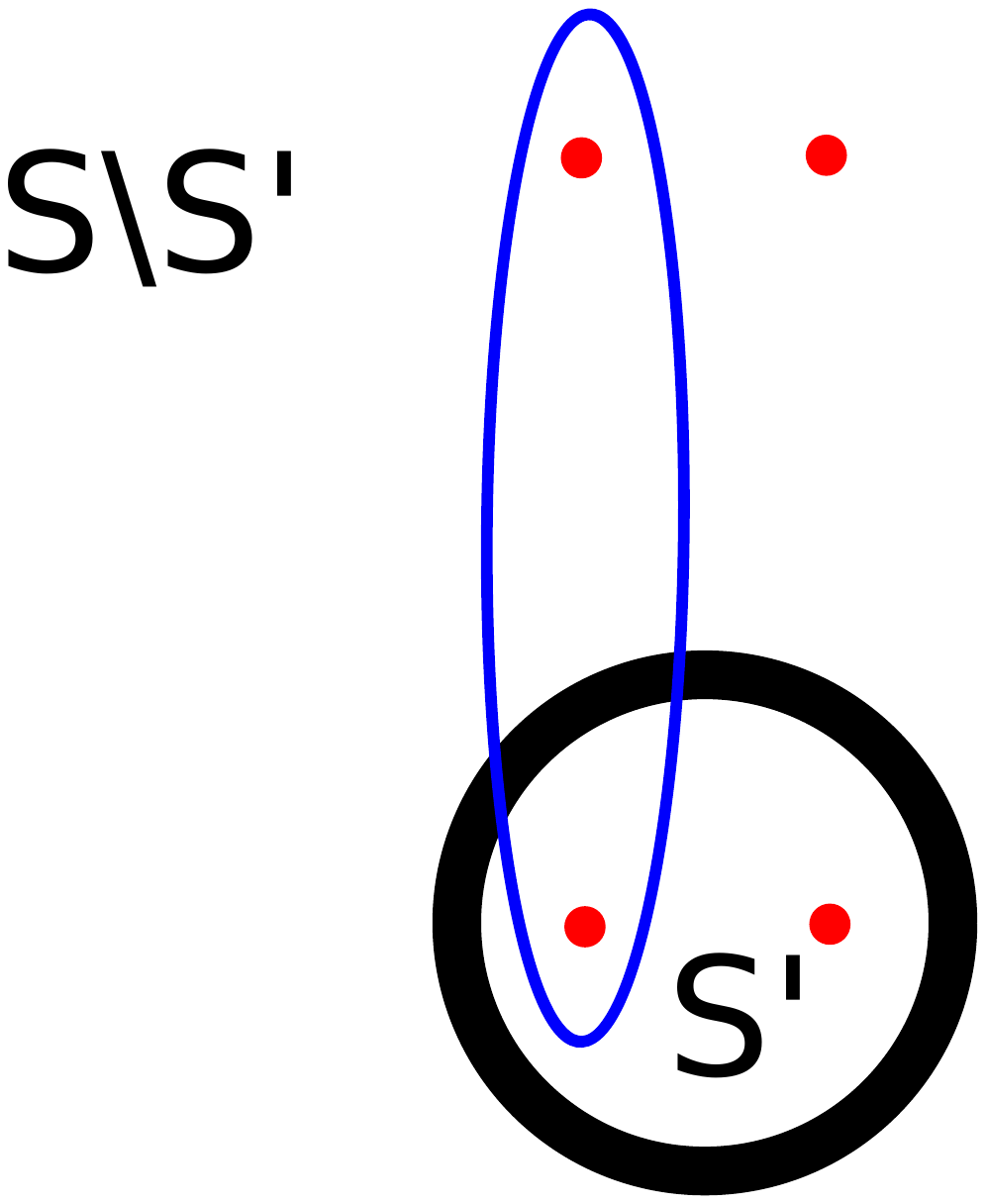}}
\end{figure}

If we ignore the multiplicities of arcs in $X\cap S'$ and $X\cap Y'$, we get a weaker notion of agreement.

\begin{definition}\rm
Let $X$ and $Y$ be two geodesic pants decompositions, and let $S'\subset S$ be a compact subsurface.  $X$ and $Y$ are said to \emph{3-agree} on $S'$ if every curve and proper arc in $X\cap S'$ is isotopic to a curve or proper arc in $Y\cap S'$, and vice versa.  In other words, $X$ and $Y$ 3-agree on $S'$ if the multicurves $X\cap S'$ and $Y\cap S'$ correspond to the same simplex in the arc-and-curve complex of $S'$.
\end{definition}

Figure \ref{3not2} shows a pair of pants decompositions which 3-agree but do not 2-agree.  

\begin{figure}
\def\svgwidth{\columnwidth}
\caption{Two pants decompositions of $S$.  The black circle is $\partial S'$, the red dots are punctures, and the blue curves are pants curves.  The two decompositions 3-agree on $S'$ but do not 2-agree on $S'$.\label{3not2}}

\subcaptionbox{\label{3not2a}}
{\includegraphics[scale=0.25]{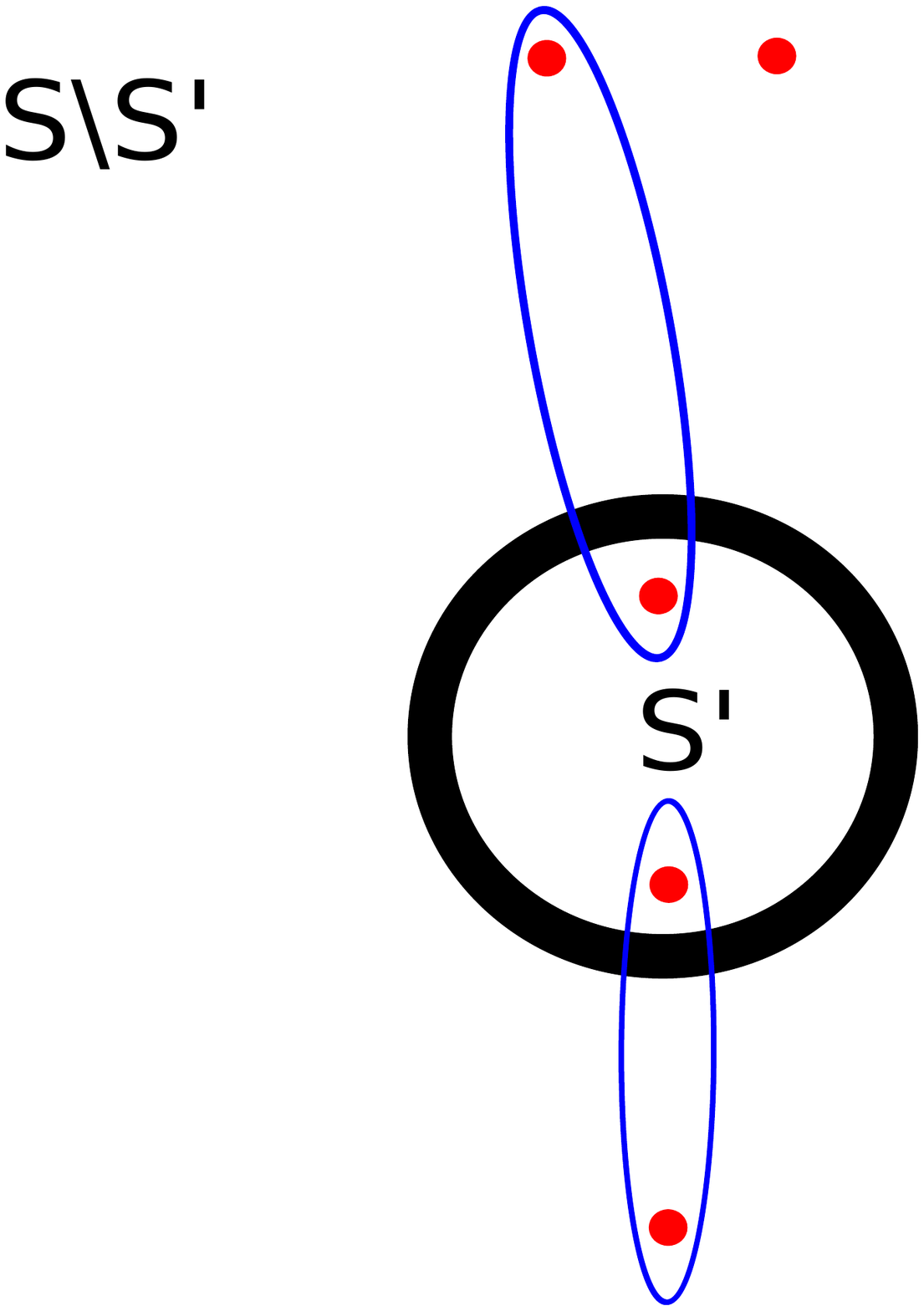}}
\subcaptionbox{\label{3not2a}}
{\includegraphics[scale=0.25]{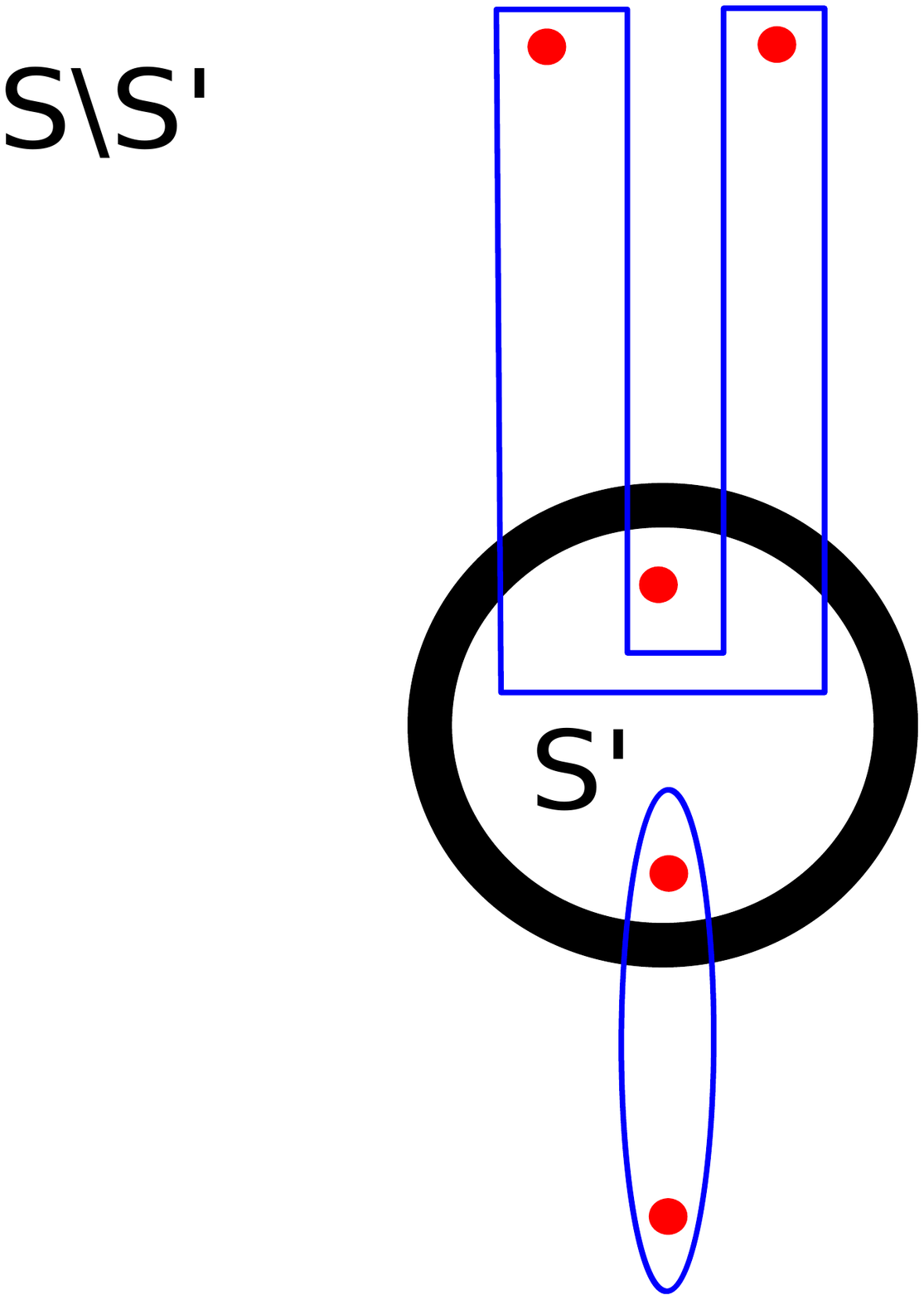}}
\end{figure}

\begin{definition}\rm
Let $X$ and $Y$ be two geodesic pants decompositions, viewed as sets of curves, and let $S'\subset S$ be a compact subsurface.  $X$ and $Y$ are said to \emph{4-agree on $Z$} if
\[\{x\in X: x\subseteq S'\}=\{y\in Y: y\subseteq S'\}\]
\end{definition}

Figure \ref{4not3} shows a pair of pants decompositions which 4-agree but do not 3-agree.  

\begin{figure}
\def\svgwidth{\columnwidth}
\caption{Two pants decompositions of $S$.  The black circle is $\partial S'$, the red dots are punctures, and the blue curves are pants curves.  The two decompositions 4-agree on $S'$ but do not 3-agree on $S'$.\label{4not3}}

\subcaptionbox{\label{4not3a}}
{\includegraphics[scale=0.25]{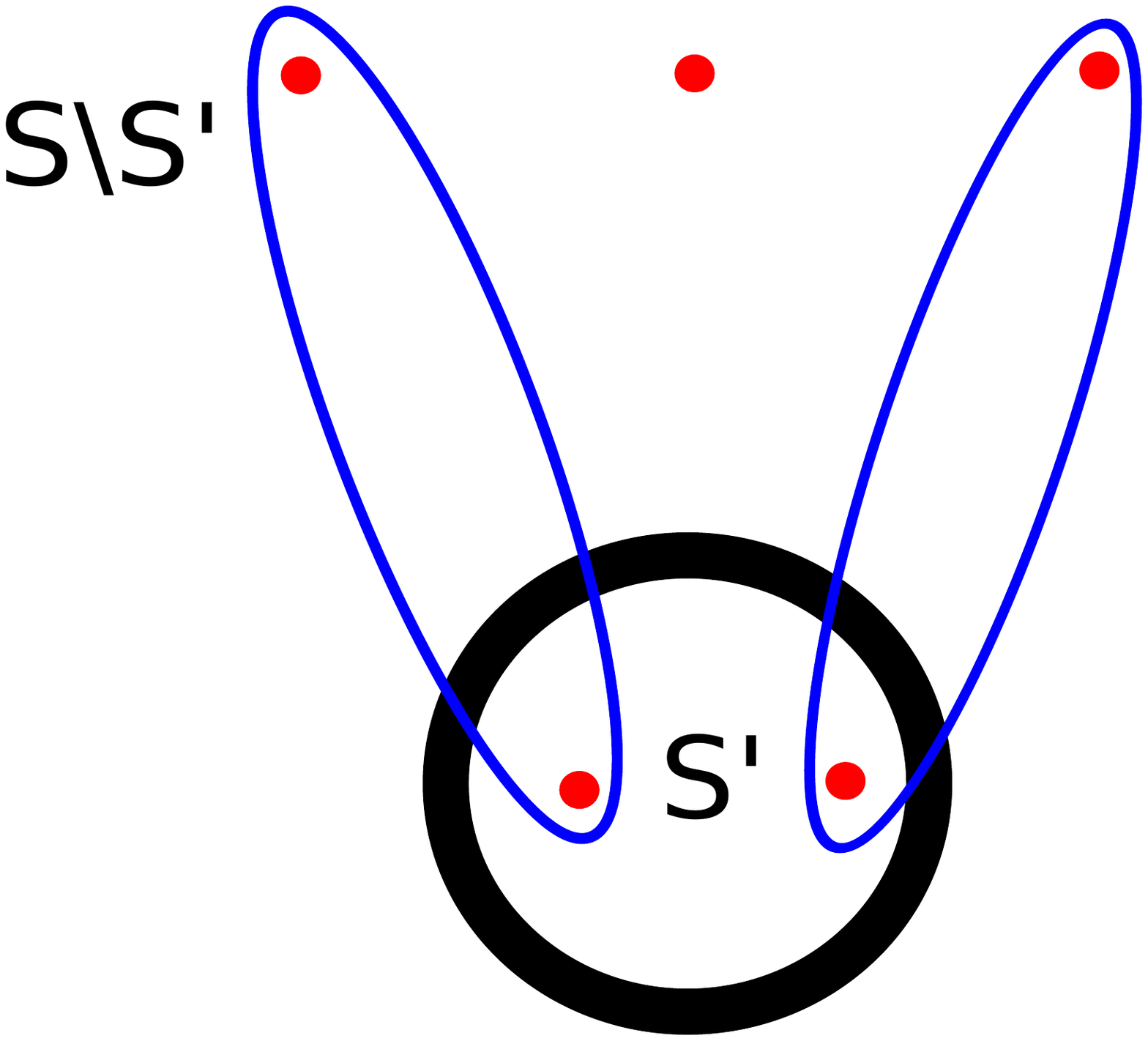}}
\subcaptionbox{\label{4not3a}}
{\includegraphics[scale=0.25]{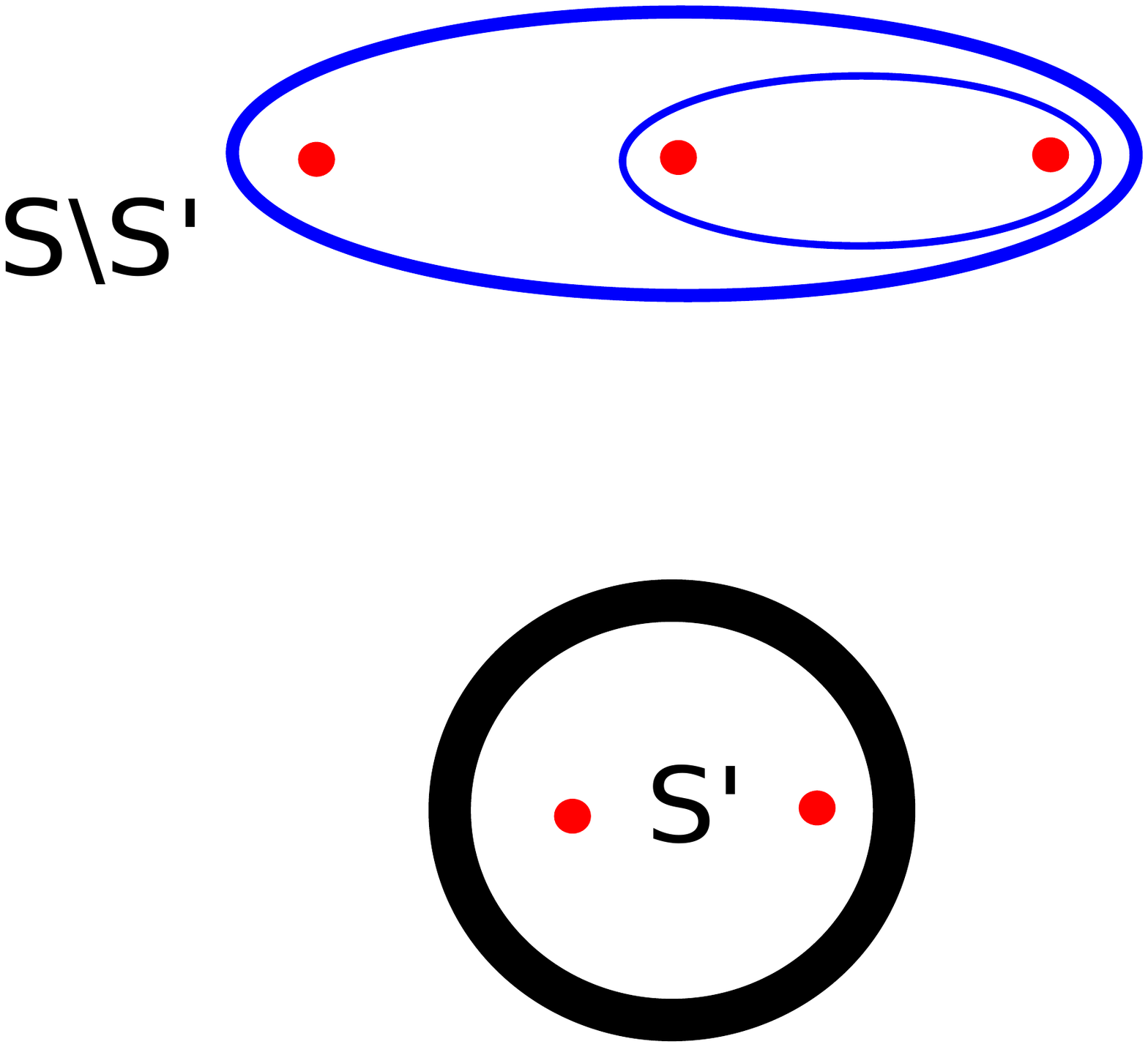}}
\end{figure}

\begin{lemma}
Let $X$ and $Y$ be two pants decompositions of $S$, and let $S'$ be a compact subsurface, and let $0\leq i\leq j\leq 4$.  If $X$ and $Y$ $i$-agree on $S'$, then they also $j$-agree on $S'$.
\end{lemma}

\begin{proof}
This result follows immediately from the definitions.
\end{proof}

\begin{lemma}
For $0\leq i\leq 4$, and for a compact subsurface $S'\subset S$, $i$-agreement on $S'$ is an equivalence relation.
\end{lemma}

\begin{proof}The result follows from the uniqueness of geodesics.\end{proof}

\begin{definition}\rm
Let $0\leq i\leq 4$ and let $X_1,X_2,\ldots$ be a sequence of pants decompositions for $S$.  The sequence is said to \emph{i-converge} to a pants decomposition $Y$ if, for all $n\in \mathbb{N}$, there exists a $m\in \mathbb{N}$ such that for all $k\geq m$, $X_k$ and $Y$ i-agree on $S_n$.
\end{definition}

Note that a sequence of pants decompositions 0-converges if and only if it is eventually constant.

\begin{definition}\rm
Let $0\leq i\leq 4$.  Define $V_i(S)$ to be the set of pants decompositions of $S$, endowed with the finest topology such that, for all sequences $\{X_j\}_{j\geq 1}$ in $V_i(S)$, $\{X_j\}_{j\geq 1}$ converges to $Y$ in $V_i(S)$ if and only if $\{X_j\}_{j\geq 1}$ $i$-converges to $Y$.
\end{definition}

The space $V_0(S)$ is just the discrete topology on the set of isotopy classes of pants decompositions of $S$.  Our goal for this section is to understand the structure of $V_i(S)$ for $1\leq i\leq 4$.

\begin{lemma}
For $1\leq i\leq 4$, the topology on $V_i(S)$ does not depend on the choice of exhaustion for $S$.  

\end{lemma}

\begin{proof}Let $S_0'\subset \cdots $ be another exhaustion for $S$.  Let $V_i'(S)$ be the topological space obtained by replacing $S_i$ with $S_i'$ in the above definition.  Suppose $X_1, X_2, \ldots$ is a sequence of pants decompositions which converge to $Y$ in $V_i(S)$.  We want to show that it also converges to $Y$ in $V_i'(S)$.  Choose some $n'\in \mathbb{N}$.  Then there is some $n\in \mathbb{N}$ such that $S_{n'}'\subseteq S_n$.  By the definition of $V_i(S)$, there exists a $m\in \mathbb{N}$ such that for all $k\geq m$, $X_k$ and $Y$ $i$-agree on $S_n$, and hence $X_k$ and $Y$ $i$-agree on $S_{n'}'$.  Hence the sequence $X_1,\ldots$ converges to $Y$ in $V_i'(S)$, so $V_i(S)$ and $V_i'(S)$ have the same topology. 
\end{proof}

\begin{definition}\rm
Suppose $X\in V_i(S)$ for some $1\leq i\leq 4$.  Let

\[
\B^i_n(X)=\{Y\in V_i(X): X\text{ and }Y\text{ i-agree on } Z_n\}
\]

\end{definition}

\begin{lemma}
For $1\leq i\leq 4$, 
\[\{\B_n^i(X): X\in V_i(S), n\in \mathbb{N}\}\]
is a basis for $V_i(S)$
\end{lemma}

\begin{proof}
First, we must show that each $B_n^i(X)$ is open in $V_i(S)$.  Suppose $Y\in B_n^i(X)$ and $X_1,\ldots $ converge to $Y$ in $V_i(S)$.  Then $\{X_j\}_{j\geq 1}$ i-converges to $Y$, hence all but finitely many of the $X_j$ are in $B_n^i(X)$.  The openness of $B_n^i(X)$ thus follows from the fact that $V_i(S)$ carries the finest possible topology such that all i-convergent sequences converge in $V_i(S)$.  
\par
Now suppose $X\in V_i(S)$ and that $U\subseteq V_i(S)$ is a neighborhood of $X$.  We want to show that there exists some $n\in \mathbb{N}$ such that $B_n^i(X)\subseteq U$.  Suppose not.  Then for all $n\in \mathbb{N}$, there exists a pants decomposition $X_n\in B_n^i(X)\backslash U$.  But this means the sequence $\{X_n\}_{n\geq 1}$ converges to $X$, but none of the $X_n$ are in $U$, a contradiction because $U$ is open.  Hence there is some $B_n^i(X)\subseteq U$.
\end{proof}

\begin{lemma}[Disjointness Lemma]
Let $X,Y\in V_i(S)$.  Then $B_n^i(X)$ are $B_n^i(Y)$ are either disjoint or equal.  
\end{lemma}

\begin{proof}
$B_n^i(X)$ and $B_n^i(Y)$ are equivalence classes of vertices with $i$-agreement on $S_n$ as the equivalence relation.  Any two equivalence classes are either disjoint or equal.
\end{proof}

\begin{definition}\rm
For $X\in V_i(S)$, let
\[V_i(S,X)=\{Y\in \mathcal{P}(S): X\text{ and }Y\text{ are in the same connected-component of } \mathcal{P}\}\]

\end{definition}

\begin{lemma}
For $1\leq i\leq 4$, $V_i(S)$ is second-countable.  In particular, for any $X\in V_i(S)$, the space has a countable basis given by
\[\{B_n^i(Y): Y\in V_i(S,X), n\in \mathbb{N}\}\]
\end{lemma}
\begin{proof}
Let $Y\in V_i(S)$ and $n\in \mathbb{N}$.  Then $B_n^i(Y)$ intersects nontrivially with $V_i(S,X)$, hence there is some $\hat{Y}\in (B_n^i(Y)\cap V_i(S,X))$.  By the Disjointness Lemma, $B_n^i(Y)=B_n^i(\hat{Y})$.  Thus,

\[\{\B_n^i(Y): Y\in V_i(S), n\in \mathbb{N}\}=\{B_n^i(Y): Y\in V_i(S,X), n\in \mathbb{N}\}\]
is a basis for $V_i(S)$.
\end{proof}

\begin{lemma}
The space $V_i(S)$ is zero-dimensional: i.e., it has a basis of clopen sets.
\end{lemma}

\begin{proof}
For each $n\geq 0$, the set 
\[\{B_n^i(Y): Y\in V_i(S)\}\]
gives a partition of $V_i(S)$ into disjoint open sets, hence each set in the basis from the previous lemma is clopen.
\end{proof}

\begin{lemma}
For $1\leq i\leq 4$, $V_i(S)$ is Hausdorff.
\end{lemma}
\begin{proof}
Let $X,Y\in V_i(S)$ with $X\neq Y$.  Choose an $n\in \mathbb{N}$ such that $X$ and $Y$ are do not $i$-agree on $S_n$.  Then by the Disjointness Lemma, $B_n^i(X)$ and $B_n^i(Y)$ are disjoint open sets that respectively contain $X$ and $Y$.
\end{proof}
\begin{lemma}
For $1\leq i\leq 4$, $V_i(S)$ is totally disconnected.
\end{lemma}
\begin{proof}
Let $X,Y\in V_i(S)$ with $X\neq Y$.  Choose an $n\in \mathbb{N}$ such that $X$ and $Y$ are do not $i$-agree on $S_n$.  Then by Lemma 4.17, $B^i_n(X)$ and $B^i_n(Y)$ are disjoint clopen sets, so $X$ and $Y$ are not in the same connected-component.
\end{proof}
Recall that a space $M$ is said to be $T_3$ if for any nonempty closed set $C\subseteq M$ and any point $m\in M\backslash C$, there exist open sets $U,V\subset M$ such that $C\subseteq U$ and $m\in V$.
\begin{lemma}
For $1\leq i\leq 4$, $V_i(S)$ is $T_3$.
\end{lemma}
\begin{proof}
Let $X\in V_i(S)$ be a pants decomposition and $C\subset V_i(S)$ be a closed set with $X\notin C$.  Note that $V_i(S)\backslash C$ is an open neighborhood of $X$.  We claim that there is an $n\in \mathbb{N}$ such that $B_n^i(X)\cap C=\emptyset$.  Suppose not.  Then for all $m\in \mathbb{N}$, there is some $Y_m\in C$ such that $Y\in B_m^i(X)$.  But this implies that $Y_1,Y_2,\ldots$ is a sequence of pants decompositions which converge to $X$, and hence this sequence must eventually enter $V_i(S)\backslash C$, which is a contradiction.  Thus, we have some $n$ such that $B_n^i(X)\cap C=\emptyset$.
\par 
Now let 
\[U=\bigcup_{Y\in C} B_n^i(Y)\]
Then $U$ is an open set containing $C$, and by the Disjointness Lemma, $B_n^i(X)\cap U=\emptyset$.
\end{proof}
\begin{corollary}
For $1\leq i\leq 4$, $V_i(S)$ is metrizable.
\end{corollary}
\begin{proof}
Combining the previous lemmas, $V_i(S)$ is a Hausdorff, $T_3$, second-countable topological space.  Hence, by the Urysohn Metrization Theorem\cite{willard2004general}, $V_i(S)$ is metrizable.
\end{proof}

\begin{lemma}
For all $X\in V_i(S)$ and all $n\geq 1$, $B^i_n(X)$ is not compact.
\end{lemma}

\begin{proof}
The set $B^i_n(X)$ has an open cover by a disjoint union of clopen sets:
\[B^i_n(X)=\bigcup_{X'\in B^i_n(X)} B^i_{n+1}(X')\]
By the assumption that $\kappa(S_{n+1}\backslash S_n)\geq 6$, this cover is infinite.  Since the sets in the cover are disjoint, they clearly cannot have a finite subcover.
\end{proof}

The following Lemma will be useful in the next section.  We state and prove it here because it is a statement about the topology of $V_i(S)$.

\begin{lemma}
Every compact subset of $V_i(S)$ has empty interior.
\end{lemma}

\begin{proof}
Suppose $C\subseteq V_i(S)$ has nonempty interior.  Then there is a point $X\in C$ and an $n\geq 1$ such that $B^i_n(X)\subset C$.  But since $B^i_n(X)$ is closed, and closed subsets of compact sets are compact, then $B^i_n(X)$ is also compact, which is a contradiction.
\end{proof}

The following Lemma will be useful in the section six.  We state and prove it here because it is a statement about the topology of $V_i(S)$.
\begin{lemma}
Let $X,Y\in V_i(S)$.  Then there exists a sequence of pants decompositions $X=X_1,X_2,\ldots$ such that $X_j$ and $X_{j+1}$ differ by an elementary move and such that the $X_j$ converge to $Y$.
\end{lemma}
\begin{proof}
Recall the assumption that the exhaustion $S_0\subset\cdots$ is \emph{principal}.  Also recall that the topology of $V_i(S)$ does not depend on the choice of principal exhaustion.  Thus, we may assume that $(\partial S_j\backslash \partial S) \subset Y$ for all $j\geq 0$.  
\par 
Since the exhaustion is principal, $S_0$ has connected pants graph.  Thus, there is a path $X_1,\ldots, X_{m_0}$ in $\mathcal{P}(S)$ such that $X_{m_0}$ agrees with $Y$ on $S_0$ and such that $(\partial S_0\backslash \partial S)\subset X_{m_0}$.  
\par 
We will inductively build a path satisfying the conclusion of this lemma.  Suppose by way of induction that there exist integers $1\leq m_0<m_1<\cdots <m_k$ and a path $X=X_1,X_2,\ldots ,X_{m_k}$ in $\mathcal{P}(S)$ such that
\begin{itemize}
\item
For all $m_j\leq \ell \leq m_k$, the pants decompositions $X_{\ell}$ and $Y$ agree on $S_j$.
\item
For all $m_j\leq \ell\leq m_k$, the pants decomposition $X_{\ell}$ contains $\partial S_j\backslash \partial S$.
\end{itemize}

Starting from $X_{m_k}$, we make additional elementary moves to extend the sequence.   The exhaustion is principal, and hence each component of $S_{k+1}\backslash S_k$ has a connected pants graph.  Thus, there is a pants decomposition $X_{m_{k+1}}$ in the same connected-component as $X_{m_k}$ such that
\begin{itemize}
\item
$(\partial S_{k+1}\backslash \partial S)\subset X_{m_{k+1}}$.
\item
The pants decompositions $X_{m_{k+1}}$ and $Y$ agree on $S_{k+1}$.
\end{itemize}

Moreover, by the induction hypothesis, $X_{m_k}$ contains $\partial S_k\backslash \partial S$, and so we can find a path from $X_{m_k}$ to $X_{m_{k+1}}$ such that all added and removed curves are entirely contained in $S\backslash S_k$.  The result follows.
\end{proof}

\section{$V_i(S)$ as a Metric Space}
In this section, we construct a metric on the set of pants decompositions of $S$, which induces the topology of $V_i(S)$ from section 4.  We will show that this metric is complete, and we will use it to prove that $V_i(S)$, as a topological space, is homeomorphic to the Baire space $\mathbb{Z}^{\mathbb{Z}}$.

\begin{definition}\rm
For $1\leq i\leq 4$, denote by $\hat{d}_i$ the \emph{partial function} on $V_i(S)\times V_i(S)$ with the following values.

\[
\hat{d}_i(X,Y)=
\begin{cases}
0 & X=Y \\
\frac{1}{n} & Y\in B^i_{n-1}(X)\backslash B^i_n(X) \\
1 & Y\notin B^i_1(X)\mbox{ and } d_0(X,Y)=1
\end{cases}
\]

\end{definition}

We will show that $\hat{d}_i$ can be extended to a distance function $d_i$ on $V_i(S)$.
\begin{lemma}
For $X,Y\in V_i(S)$, $\hat{d}_i(X,Y)=0$ if and only if $X=Y$.
\end{lemma}
\begin{proof}
This lemma follows immediately from the definition.
\end{proof}

\begin{lemma}
Suppose $\hat{d}_i(X,Y)$ is defined.  Then $\hat{d}_i(X,Y)=\hat{d}_i(Y,X)$.
\end{lemma}
\begin{proof}
If $Y\in B^i_n(X)$ for some $n\geq 0$, then the Disjointness Lemma implies that $X\in B^i_n(Y)$, hence $\hat{d}_i(X,Y)=\hat{d}_i(Y,X)$.  If $Y\notin B^i_0(X)$ but $d_0(X,Y)=1$, then $d_0(Y,X)=1$ because $d_0$ is symmetric, and hence $\hat{d}_i(X,Y)=\hat{d}_i(Y,X)=1$.
\end{proof}
\begin{lemma}
Suppose $X$,$Y$, and $Z$ are distinct vertices such that $\hat{d}_i(X,Y)$, $\hat{d}_i(X,Z)$, and $\hat{d}_i(Y,Z)$ are all defined.  Then $\hat{d}_i(X,Z) < \max(\{\hat{d}_i(X,Y),\hat{d}_i(Y,Z)\})$
\end{lemma}
\begin{proof}
First, suppose $\hat{d}_i(X,Y)=1$.  Since $\hat{d}_i$ never takes values greater than 1, we have 
\[\hat{d}_i(X,Z)\leq 1=\hat{d}_i(X,Y)=\max(\{\hat{d}_i(X,Y),\hat{d}_i(Y,Z)\})\]
By a similar argument, if $\hat{d}_i(Y,Z)=1$, then 
\[\hat{d}_i(X,Z)\leq 1=\hat{d}_i(Y,Z)=\max(\{\hat{d}_i(X,Y),\hat{d}_i(Y,Z)\})\]
\par 
Now suppose that $\hat{d}_i(X,Y)$ and $\hat{d}_i(Y,Z)$ are both less than 1.  Then there exist integers $m,n\geq 1$ such that $B^i_m(X)=B^i_m(Y)$ and $B^i_n(Y)=B^i_n(Z)$.  Without loss of generality, we will assume $m\leq n$.  Then $Z\in B^i_n(Y)\subseteq B^i_m(Y)=B^i_m(X)$, hence $\hat{d}_i(X,Z)\leq \hat{d}_i(X,Y)\leq \max(\{\hat{d}_i(X,Y),\hat{d}_i(Y,Z)\})$.  
\end{proof}

\begin{theorem}
There is a distance function $d_i$ on the set of pants decompositions of $S$ such that 
\begin{enumerate}
\item
Whenever $\hat{d}_i(X,Y)$ is defined, $\hat{d}_i(X,Y)=d_i(X,Y)$.
\item 
If $\hat{d}_i(X,Y)$ is undefined, then $d_i(X,Y)\geq 1$.  

\item
The topology induced by $d_i$ is the same as the topology on $V_i(S)$ defined in section 4.
\end{enumerate}
\end{theorem}

\begin{proof}
To construct $d_i$, consider the set $\mathbb{S}_i(X,Y)$ of all finite sequences of vertices $X=X_0,X_1,\ldots , X_{\ell}=Y$ from $X$ to $Y$ such that $\hat{d}_i(X_j,X_{j+1})$ is defined for all $0\leq j<\ell$.  Let 
\[d_i(X,Y)=\inf_{\mathbb{S}_i(X,Y)}(\hat{d}_i(X_0,X_1)+\hat{d}_i(X_1,X_2)+\cdots \hat{d}_i(X_{\ell-1},X_{\ell})\]
To show that $d_i(X,Y)$ is well-defined, we must show that $\mathbb{S}_i(X,Y)$ is nonempty for all $X,Y\in V_i(S)$.  For $n\geq 0$, $B^i_n(X)$ has nonempty intersection with every connected-component of $\mathcal{P}(S)$.  Thus, there is some $X_1\in B^i_0(X)$ such that $X_1$ and $Y$ are in the same connected component.  Hence, there exists a sequence $(X=X_0,X_1,\ldots, X_{\ell}=Y)$ such that $X_j$ is adjacent to $X_{j+1}$ for $1\leq j<\ell$, and hence $\hat{d}_i(X_j,X_{j+1})$ is defined.  Thus $\mathbb{S}_i(X,Y)$ is nonempty, so we can take the infimum of a nonnegative function over it.  

\par 
The fact that $d_i$ is symmetric follows from the symmetry of $\hat{d}_i$ and the construction of $d_i$.  Similarly, $d_i$ satisfies the triangle inequality because $\hat{d}_i$ does when it is defined. 
\par 
If $\hat{d}_i(X,Y)$ is defined, then $\mathbb{S}_i(X,Y)$ contains the two vertex sequence $(X,Y)$.  No other sequence can have a shorter length by Lemma 5.4.  Thus $\hat{d}_i(X,Y)=d_i(X,Y)$.
\par 
Next suppose $\hat{d}_i(X,Y)$ is undefined, and consider a sequence $(X,X_1,\ldots,X_{\ell}=Y)\in \mathbb{S}_i(X,Y)$.  Since $\hat{d}_i(X,Y)$ is undefined, we must have $\ell\geq 2$.  We claim that there exists a $0\leq j<\ell$ such that $\hat{d}_i(X_j,X_{j+1})=1$, and hence $\sum_{j=0}^{\ell-1} \hat{d}_i(X_j,X_{j+1})>1$.  Suppose not.  Then $X_{j+1}\in B^i_1(X_j)$ for all $j$, so by the Disjointness Lemma $Y\in B^i_1(X)$, so $\hat{d}_i(X,Y)$ is defined, which is a contradiction.  Thus, $d_i(X,Y)\geq 1$ whenever $\hat{d}_i(X,Y)$ is undefined.  In particular, $d_i(X,Y)=0$ if and only if $X=Y$, so $d_i$ is actually a distance function.
\par 
Finally, we must show that the metric given by $d_i$ is compatible with the topology on $V_i(S)$ defined in section 4.  Note that for $0<\epsilon<1$, the open $\epsilon$-ball around a point $X$ is $B^i_{\lfloor \frac{1}{\epsilon}\rfloor}(X)$.  Hence, the topology induced by $d_i$ has a basis
\[\{B^i_n(X):X\in V_i(S),n\geq 1\}\]
which is also a basis for the previously defined topology on $V_i(S)$.  
\end{proof}

We also have the following.
\begin{lemma}
$V_i(S)$ is a complete metric space.
\end{lemma}

\begin{proof}
Let $X_1,\ldots$ be a Cauchy sequence in $V_i(S)$.  Let $n\geq 2$.  Then by the definition of a Cauchy sequence, there exists an $N\in \mathbb{N}$ such that $d_i(X_j,X_k)<\frac{1}{n}$ whenever $j,k\geq N$.  But by construction, this means that $X_j$ and $X_k$ agree on $S_n$.  Hence, all but a finite number of terms in the sequence agree on any finite-type subsurface of $S$, and hence the sequence converges.
\end{proof}

The preceeding lemma implies that, as a topological space, $V_i(S)$ is completely metrizable.  We now recall the \emph{Baire Space}, which is the countably infinite Cartesian product of a countable infinite discrete space.  The Baire set is also homeomorphic to $\mathbb{R}\backslash \mathbb{Q}$ with the usual topology.  The Alexandrov-Urysohn Theorem characterizes when a space is homeomorphic to the Baire Space.

\begin{theorem}[\cite{Kechris_book}, Theorem 7.7]
A space is homeomorphic to the Baire space if and only if it is non-empty, zero-dimensional, seperable, completely metrizable, and if every compact subset has empty interior.
\end{theorem}

\begin{corollary}
For any infinite-type surface $S$, $V_i(S)$ is homeomorphic to the Baire space.
\end{corollary}
\begin{proof}
Any infinite-type surface has a pants decomposition, hence $V_i(S)$ is nonempty.  In section 4, we have shown that $V_i(S)$ is seperable, zero-dimensional, and that all compact subsets of $V_i(S)$ have empty interior. The previous lemma shows that $V_i(S)$ is completely metrizable.
\end{proof}

\begin{corollary}
For a pants decomposition $X\in V_i(S)$ and an integer $n\geq 0$, the subspace $B^i_n(X)\subset V_i(S)$ is homeomorphic to the Baire space.
\end{corollary}

\begin{proof}
The subspace $B^i_n(X)$ is a closed subset of a completely metrizable space, and hence is completely metrizable.  The space $B^i_n(X)$ has a basis
\[\{B^i_m(Y):Y\in B^i_n(X), m\geq n\}
\]
We can show that $B^i_n(X)$ is zero-dimensional and seperable and such that all compact subsets have empty interior using the same arguments as for $V_i(X)$.  
\end{proof}

\begin{remark}\rm
By Corollary 9 of \cite{vlamisnotes}, $\Mod(S)$ with the compact-open topology is homeomorphic to a Baire space, and hence to $V_i(S)$.  This phenomenon also occurs in the finite-type setting: for a finite-type hypberbolic surface $S$, the zero skeleton of $\mathcal{P}(S)$ and $\Mod(S)$ with the compact-open topology are both countable discrete spaces.
\end{remark}

The goal of the next two lemmas is to sharpen part (2) of Theorem 5.5.

\begin{lemma}
Let $n\geq 3$, and let $X$ and $Y$ be pants decompositions such that $1<d_i(X,Y)< 1+\frac{1}{n}$.  Then there exist pants decompositions $X'$ and $Y'$ such that
\begin{itemize}
\item
$X'\in B^i_n(X)$.
\item
$Y'\in B^i_n(Y)$.
\item
$X'$ and $Y'$ disagree on $S_0$.
\item
$X'$ and $Y'$ differ by an elementary move.
\end{itemize}
\end{lemma}

\begin{proof}
By the definition of $d_i$, there exists a sequence of pants decompositions 
\[(X_0,\ldots , X_{\ell})\in \mathbb{S}_i(X,Y)\]
such that 
\[\sum_{j=1}^{\ell}(\hat{d}_i(X_{j-1},X_j))< 1+\frac{1}{n}\]

Our goal is to modify this sequence to obtain a sequence $(X,X',Y',Y)\in \mathbb{S}_i(X,Y)$ such that $X'$ and $Y'$ have the desired properties. 
\par 
Call an element of $\mathbb{S}_i(X,Y)$ \emph{reduced} if, for all $0\leq j_1<j_2\leq \ell$, $\hat{d}_i(X_{j_1},X_{j_2})$ is defined if and only if $j_2-j_1=1$.  
\par 

Suppose that $(X_0,\ldots,X_{\ell})$ is not reduced: i.e., there exist indices $0\leq j_1<j_1+2\leq j_2\leq \ell$ such that $\hat{d}_i(X_{j_1},X_{j_2})$ is defined.  Then we can remove all pants decompositions in between $X_{j_1}$ and $X_{j_2}$ to obtain a shorter sequence,
\[(X_0,\ldots, X_{j_1},X_{j_2},\ldots X_{\ell})\in \mathbb{S}_i(X,Y).\]  By nonnegativity of $\hat{d_i}$, the sum of the distances of this sequence is less than that of the original sequence.  Hence, we can repeat this process up to $\ell$ times until we get a reduced element of $\mathbb{S}_i(X,Y)$.  Thus, we may assume $(X_0,\ldots, X_{\ell})$ is reduced.
\par 
If $X_j$ and $X_{j+1}$ disagree on $S_0$, then $\hat{d}_i(X_j,X_{j+1})=1$. But $\sum_{j=1}^{\ell}(\hat{d}_i(X_{j-1},X_j))<2$. Hence there can be at most one index $j$ such that $X_j$ and $X_{j+1}$ disagree on $S_0$.
\par 
Now consider an index $0\leq j\leq \ell-2$.  Since our sequence is reduced, $X_j$ and $X_{j+2}$ disagree on $S_0$.  Since $i$-agreement on $S_0$ is an equivalence relation, the previous sentence implies that either $X_j$ and $X_{j+1}$ disagree on $S_0$ or $X_{j+1}$ and $X_{j+2}$ disagree on $S_0$.  
\par 
Combining the results of the previous two paragraphs, we deduce that $\ell\leq 3$.  We tackle each possible value of $\ell$.

\par 
If $\ell=3$, then $X_1$ and $X_2$ must disagree on $S_0$ in order for the sequence to be reduced (otherwise three consecutive $X_j$ would agree on $S_0$.) Thus $X_1$ and $X_2$ differ by an elementary move.  We now claim that $X'=X_1$ and $Y'=X_2$ satisfy the conclusion of the lemma.  We have already seen that $X'$ and $Y'$ disagree on $S_0$ and differ by an elementary move.  Also, by hypothesis, 
\[1+\frac{1}{n}> \hat{d}_i(X_0,X_1)+\hat{d}_i(X_1,X_2)+\hat{d}_i(X_2,X_3)=1+\hat{d}_i(X_0,X_1)+\hat{d}_i(X_2,X_3)\]
Hence, $\hat{d}_i(X_0,X_1)<\frac{1}{n}$ and $\hat{d}_i(X_2,X_3)<\frac{1}{n}$, so by definition $X_1\in B^i_n(X_0)=B^i_n(X)$ and $X_2\in B^i_n(X_3)=B^i_n(Y)$, as desired.
\par 
Now we turn our attention to the case $\ell=2$.  Either $X_0$ and $X_1$ disagree on $S_0$, or $X_1$ and $X_2$ disagree on $S_0$.  If $X_0$ and $X_1$ disagree on $S_0$, then we must have $X_1\in B^i_n(X_2)=B^i_n(Y)$ by the same argument as the $\ell=3$ case.  Thus, we can let $X_1=Y'$ and $X=X_0=X'$, and we are done.  Similarly, if $X_1$ and $X_2$ disagree on $S_0$, then $X_1\in B^i_n(X)$, so we can set $X_1=X'$ and $X_2=Y=Y'$, and we are done.
\par
Finally, we note that since $\hat{d}_i(X,Y)$ is undefined, we must have $\ell\geq 2$, so we don't need to consider $\ell=1$ or $\ell=0$.
\end{proof}

\begin{lemma}
If $X,Y\in V_i(S)$, then $\hat{d}_i(X,Y)$ is undefined if and only if $d_i(X,Y)>1$.
\end{lemma}
\begin{proof}
If $\hat{d}_i(X,Y)$ is defined, then Theorem 5.5 (1) immediately tells us that $d_i(X,Y)=\hat{d}_i(X,Y)\leq 1$.  
\par 
Now suppose $\hat{d}_i(X,Y)$ is undefined.  We will tackle this lemma in two cases, depending on the cardinality of $X\backslash Y$.
\par 
First, consider the case where $\vert X\backslash Y\vert =1$, i.e., there is a unique curve $\alpha$ in $X$ but not in $Y$.  Let $S_{\alpha}$ denote the complexity 1 subsurface contained in the complement of the union of $X\backslash\{\alpha\}$. 
\par 
Note that $\vert Y\backslash X\vert =1$, and let $\alpha'$ the unique curve in $Y\backslash X$.  Observe that $\alpha'\subset S_{\alpha}$, because $X\backslash \alpha=Y\backslash \alpha'$ has only one component of positive complexity.
\par 
Since $S_{\alpha}$ has finite type, there exists a natural number $n$ such that $\overline{S_{\alpha}}\subset S_n$.  We claim that $d_i(X,Y)\geq 1+\frac{1}{n}$.  Suppose, by way of contradiction, that $d_i(X,Y)< 1+\frac{1}{n}$.  Then by the previous lemma, there exist pants decompositions $X'$ and $Y'$ such that
\begin{itemize}
\item
$X'\in B^i_n(X)$.
\item
$Y'\in B^i_n(Y)$.
\item
$X'$ and $Y'$ differ by an elementary move.
\end{itemize}

Since $X$ and $X'$ agree on $S_n$, and $\alpha\subset S_{\alpha}\subset S_n$, we must have $\alpha\in X'$.  By the same argument, $\alpha'\in Y'$.  Thus, the elementary move that takes $X'$ to $Y'$ must replace $\alpha$ by $\alpha'$.  Hence, $\alpha$ and $\alpha'$ intersect minimally.  But this statement is impossible, because it implies that $X$ and $Y$ differ by an elementary move, and hence $\hat{d}_i(X,Y)$ is defined.  Thus, we can conclude $d_i(X,Y)\geq 1+\frac{1}{n}>1$.
\par 
Now, consider the case where $\vert X\backslash Y\vert >1$, and choose two distinct curves $\alpha,\beta\in X\backslash Y$.  As in the previous case, we let $S_{\alpha}$ denote the complexity 1 subsurface contained in the complement of $X\backslash \alpha$.  Similarly, we let $S_{\beta}$ denote the complexity 1 subsurface contained in the complement of $X\backslash \beta$.  Note that since $\alpha$ and $\beta$ are disjoint and nonisotopic, $S_{\alpha}\neq S_{\beta}$, although the subsurfaces are not necessarily disjoint.
\par 
We may pass to the geodesic representatives of $X$ and $Y$, and thus  we can assume that $S_{\alpha}$ and $S_{\beta}$ have geodesic boundaries.  Hence $S_{\alpha}\cup S_{\beta}$ is of finite type, so there exists a natural number $n$ such that $S_{\alpha}\cup S_{\beta}\subset S_n$.  As in the previous case, suppose that $d_i(X,Y)<1+\frac{1}{n}$.  Then by the previous lemma, there exist pants decompositions $X'$ and $Y'$ such that
\begin{itemize}
\item
$X'\in B^i_n(X)$.
\item
$Y'\in B^i_n(Y)$.
\item
$X'$ and $Y'$ differ by an elementary move.
\end{itemize}

Since $\alpha$ and $\beta$ are contained in $S_{\alpha}\cup S_{\beta}\subset S_n$, and since $X$ and $X'$ agree on $S_n$, we must have $\alpha, \beta\in X'$. Similarly, $Y$ and $Y'$ agree on $S_n$, so $\alpha,\beta\notin Y'$.  Thus $X'\backslash Y'$ contains at least two curves.  But this is impossible since $X'$ and $Y'$ differ by an elementary move.  Thus we can conclude that $d_i(X,Y)\geq 1+\frac{1}{n}>1$. 

\end{proof}

Recall that a metric space $(M,d)$ is said to be an \emph{ultrametric space} if for all $x,y,z\in M$, $d(x,z)\leq \max(\{d(x,y),d(y,z)\})$.  With the usual Euclidean metric, $\mathbb{R}^n$ is not an ultrametric space, because any three distinct colinear points violate the ultrametric inequality. 

\begin{corollary}
If $U\subset V_i(X)$ is a subspace of diameter $\leq 1$, then $U$ with the restriction of $d_i$ is an ultrametric space.
\end{corollary}

\begin{proof}
By the previous lemma, $\hat{d}_i$ is defined on all of $U\times U$, and hence $d_i(x,y)=\hat{d}_i(x,y)$ for all $x,y\in U$.  The result now follows from Lemma 5.4.
\end{proof}

\begin{corollary}
The metric $d_i$ is not an ultrametric on $V_i(S)$.
\end{corollary}

\begin{proof}
Let $X,Y\in V_i(S)$ such that $X$ and $Y$ disagree on $S_0$ and such that $X$ and $Y$ do not differ by an elementary move.  Then $\hat{d}_i(X,Y)$ is undefined, and hence $d(X,Y)>1$ by the previous lemma.  But by Theorem 5.5, there exists a sequence of points $X=X_0,X_1,\ldots,X_{\ell}=Y$ such that $d_i(X_j,X_{j+1})=\hat{d}_i(X_j,X_{j+1})\leq 1$ for all $0\leq j<\ell$.  If $d_i$ were an ultrametric, then we would have $d_i(X,Y)\leq 1$, a contradiction.
\end{proof}
\begin{comment}

\end{comment}
\section{The spaces $\mathcal{PS}_i(S)$}
The family of spaces $V_i(S)$ defined in the previous section are totally disconnected.  This fact is not very surprising, because $V_i(S)$ contains only the vertex set of $\mathcal{P}(S)$, without the edges.  In this section, using the definitions of $V_i(S)$ as a template, we will define a topology on the entire pants graph (including the edges), and call this space $\mathcal{PS}_i(S)$.  The topology on $\mathcal{PS}_i(S)$ will be weaker than the usual topology on $\mathcal{P}(S)$, but it will still encode all of the structure of the pants graph. In this section, we will construct a basis of open sets for $\mathcal{PS}_i(S)$, and we will prove it is second-countable.

\par 
$\mathcal{PS}_i(S)$ will be equal as a set to $\mathcal{P}(S)$.  We think of $\mathcal{P}(S)$ as the set of vertices with copies of the unit interval glued to each pair of adjacent vertices.  We can therefore represent a point in $\mathcal{PS}_i(S)$ by a three-tuple $(X,a,Y)$, where $X$ and $Y$ are adjacent vertices in $\mathcal{P}(S)$ and $a\in [0,1]$ is the distance between the point and $X$.  Each edge point has two such represenations, since $(X,a,Y)=(Y,1-a,X)$.  Also, each vertex is represented by infinitely many three-tuples:
\[X=(X,0,Y)=(Y,1,X)\]
for any vertex $X$ and any vertex $Y$ adjacent to $X$.  
\par 
Definition 6.1 states that a sequence converges in $\mathcal{PS}_i(S)$ if it satisfies any of three conditions.  Intuitively, the first two conditions are rough analogs of the fact that a sequence converges in a product space converges if and only if the projections of the sequence to each coordinate converge.  
\par 
If we think of edges as analogous to real intervals bounded by their incident vertices, then the third condition is an analog of the Squeeze Theorem.    
\par 

We are now ready to define $i$-converges for $\mathcal{PS}_i(S)$.

\begin{definition}\rm
Let $\{P_k\}_{k\geq 1}$ be a sequence of points (either vertex or edge points) in $\mathcal{P}(S)$.  For $0\leq i\leq 4$, the sequence is set to \emph{i-converge} to a point $P$ if there exist representatives $P_k=(X_k,a_k,Y_k)$ for $k\in \mathbb{N}$ and $P=(X,a,Y)$ such that \emph{at least one} of the following conditions hold:
\begin{enumerate}
\item
$\{X_k\}\to X$ and $\{Y_k\}\to Y$ in $V_i(S)$, and $\{a_k\}\to a$ in $[0,1]$.

\item
$P$ is a vertex, $\{X_k\}\to P$ in $V_i(S)$, and $\{a_k\}\to 0$ in $[0,1]$.

\item
$P$ is a vertex, and $\{X_k\}$ and $\{Y_k\}$ both converge to $P$ in $V_i(S)$
\end{enumerate}
\end{definition}
\begin{definition}\rm
For $0\leq i\leq 4$, let $\mathcal{PS}_i(S)$ be the set $\mathcal{P}(S)$ equipped with the finest topology such that any $i$-convergent sequence converges in $\mathcal{PS}_i(S)$.
\end{definition}

Let $\Gamma\subset \mathcal{P}(S)$ be a connected component.  Recall that $\Gamma$ can be made into a metric space by letting each edge have length 1.

\begin{lemma}
For any surface $S$, $\mathcal{PS}_0(S)\cong \mathcal{P}(S)$ with topology induced by the graph distance metric.
\end{lemma}
\begin{proof}
We will show that a sequence converges in $\mathcal{PS}_0(S)$ if and only if it converges in $\mathcal{P}(S)$.  By construction, any convergent sequence in $\mathcal{P}(S)$ also converges in $\mathcal{PS}_0(S)$.  Now choose a 0-convergent sequence $\{(X_k,a_k,Y_k\}\to (X,a,Y)$ in $\mathcal{PS}_0(S)$.  Recall that $V_0(S)$ carries the discrete topology.  Thus, condition (1) in definition 6.1 can occur only if $\{X_k\}$ and $\{Y_k\}$ are eventually constant.  In this case it follows that $\{(X_k,a_k,Y_k)\}$ eventually agrees with $\{X,a_k,Y\}$, which obviously converges to $\{X,a,Y\}$ in $\mathcal{P}(S)$.
\par 
Now suppose that $\{(X_k,a_k,Y_k\}\to P$ in $\mathcal{PS}_0(S)$ by condition 2 in definition 6.1.  Then $\{X_k\}$ is eventually constant and converges to $P$.  Also, $\{a_k\}\to 0$, so $\{P,a_k,Y_k\}$ converges to $P$ in $\mathcal{P}(S)$.
\par 
Condition 3 in definition 6.1 cannot occur in the case $i=0$.  This is because $V_0(S)$ is discrete, so condition 3 says that $\{X_k\}$ and $\{Y_k\}$ are eventually equal.  But $C_P(S)$ is a simple graph, so $X_k=Y_k$ cannot be adjacent to itself, which is a contradiction.
\par 
Thus, in all three cases, a convergent sequence in $\mathcal{PS}_0(S)$ also converges to the same limit in $\mathcal{P}(S)$, so $\mathcal{PS}_0(S)\cong \mathcal{P}(S)$.
\end{proof}

\begin{lemma}
For $0\leq i\leq j\leq 4$, the topology on $\mathcal{PS}_i(S)$ is finer than the topology on $\mathcal{PS}_j(S)$.  
\end{lemma}
\begin{proof}
This lemma follows immediately from the fact that $V_i(S)$ has a finer topology than $V_j(S)$ for $0\leq i\leq j\leq 4$.
\end{proof}

\begin{lemma}
For $0\leq i\leq 4$, the restriction of the topology on $\mathcal{PS}_i(S)$ to the vertex set agrees with that of $V_i(S)$.
\end{lemma}
\begin{proof}
Immediate from the definition.
\end{proof}

Our next goal is to construct a basis for $\mathcal{PS}_i(S)$ that is relatively easy to work with.  Intuitively, we can think of points in $\mathcal{PS}_i(S)$ as close if they are either close in the underlying graph, or if they agree on a large finite-type subsurface.  Thus far, we lack a definition of two points "agreeing on $S_n$ if one or both of the points is an edge point.  To this end, will define sets $A^i_n(P)$ for each $P\in \mathcal{PS}_i(S)$.  The $A^i_n(P)$ will not be open, but we will use them to construct open neighborhoods of each point in $\mathcal{PS}_i(S)$.  The definition of $A^i_n(P)$ will be different depending on whether $P$ is a vertex point or an edge point.

\begin{definition}\rm
If $X\in \mathcal{PS}_i(S)$ is a vertex, let $A^i_n(X)$ be the induced subgraph on $B^i_n(X)$.  
\par 
If $P=(X,a,Y)$ is an edge point, let 
\[A^i_n(P)=\{(X',a,Y'):X'\in B^i_n(X)\text{ and }Y'\in B^i_n(Y)\}\]
\end{definition}
\par 
If $P$ is a vertex, then $A^i_n(X)$ can be thought of as the "cloud" of vertices and edges near $X$.  If $P$ is an edge point, then $A^i_n(X)$ can be thought of as a "horizontal slice" of the edges near $P$.
\par 
Note that $A^i_n(P)$ need not be open in $\mathcal{PS}_i(S)$.  Also note that for the case $i=0$, $A^0_n(P)$ is just a single point $P$.

\begin{lemma}[Vertex Disjointness Lemma]
Let $X$ and $Y$ be two vertices in $\mathcal{PS}_i(S)$.  Then $A^i_n(X)$ and $A^i_n(Y)$ are either disjoint or equal.
\end{lemma}
\begin{proof}
Follows immediately from the disjointness lemma for $V_i(S)$.
\end{proof}

\begin{lemma}[Edge Disjointness Lemma]
Let $P$ and $Q$ be two edge points in $\mathcal{PS}_i(S)$.  Then $A^i_n(P)$ and $A^i_n(Q)$ are either disjoint or equal.
\end{lemma}
\begin{proof}
Follows immediately from the disjointness lemma for $V_i(S)$.
\end{proof}
Note that if $X$ is a vertex and $P$ is a point on an edge incident to $X$, then disjointness between $A^i_n(X)$ and $A^i_n(P)$ can fail.  If $P=(X,a,Y)$ is an edge point, and if $X$ $i$-agrees on $S_n$ with $Y$, then $A^i_n(P)$ is a proper subset of $A^i_n(X)=A^i_n(Y)$.
\par 
We need one more family of auxillary sets, $D^i_{\epsilon,n}(P)$.  Starting from $P$, we move a distance of up to $\epsilon$ in the graph metric on $\mathcal{P}(S)$ until we reach another point $P'$.  From there, we jump to another point which agrees with $P'$ on $S_n$: i.e., a point in $A^i_n(P')$.  The set of all possible ending points of these two steps is $D^i_{\epsilon,n}(P)$.
\par 

If $P_1$ and $P_2$ are points in a connected-component of $\mathcal{P}(S)$, we denote by $d_0(P_1,P_2)$ their distance in the graph metric.  If $P_1$ and $P_2$ are in different components, then we let $d_0(P_1,P_2)=\infty$. We can now restate the previous two paragraphs as a formal definition.
\begin{definition}\rm
Let $P\in \mathcal{PS}_i(S)$, $n\in \mathbb{N}$, and $\epsilon>0$.  Then let
\[D^i_{\epsilon,n}(P)=\bigcup_{P':d_0(P,P')<\epsilon} A^i_n(P')\]
\end{definition}

Now we are ready to define the basis of open sets for $\mathcal{PS}_i(S)$.

\begin{definition}\rm
Let $X\in \mathcal{PS}_i(S)$ be a vertex.  Let $0\leq i\leq 4$ and $n\geq 0$ be integers and let $\epsilon>0$ be a real number.  Let
\[\mathbb{D}^i_{\epsilon,n}(X)=\bigcup_{X':X'\mbox{ and }X\mbox{ }i\mbox{-agree on } S_n} D^i_{\epsilon,n}(X)\].
\end{definition}

\begin{definition}\rm
If $P=(X,a,Y)$ and $Q=(X',a',Y')$ are edges points, we say that $P$ and $Q$ \emph{$i$-agree on $S_n$} if $X$ and $X'$ $i$-agree on $S_n$, $Y$ and $Y'$ $i$-agree on $S_n$, and $a=a'$.  As in the case of vertices, we will sometimes say that $P$ and $Q$ \emph{agree on $S_n$} if $i$ is understood.
\end{definition}

\begin{definition}\rm
Let $P$ be an edge point.  Let $0\leq i\leq 4$ and $n\geq 0$ be integers and let $\epsilon>0$ be a real number.  Let
\[\mathbb{D}^i_{\epsilon,n}(P)=\bigcup_{P \mbox{ and } P'\mbox{ agree on }S_n}D^i_{\epsilon,n}(P'))\].
\end{definition}

It will sometimes be convenient to have a basis indexed only by a single positive real number.  Thus, we define

\[\mathbb{D}^i_{\epsilon}(P)=\mathbb{D}^i_{\epsilon, \lfloor \frac{1}{\epsilon}\rfloor}(P)\]

Note that for $i=0$, $\mathbb{D}^0_{\epsilon,n}(P)$ is the open ball of radius $\epsilon$ around $P$ in the graph-distance metric on $\mathcal{P}(S)$.

\par 
The goal for the rest of this section is to prove the following theorem.

\begin{theorem}[Basis Theorem]
For $1\leq i\leq 4$, the set
\[\{\mathbb{D}^i_{\epsilon}(P):\epsilon>0,P\in \mathcal{PS}_i(S)\}\]
is a basis for $\mathcal{PS}_i(S)$. 
\end{theorem}

We will prove this theorem by breaking it up into several lemmas.

\begin{lemma}
Let $P=(X,a,Y)$ be an edge point, $n\in \mathbb{N}$, and $0<\epsilon<\max({a,1-a})$.  Then $\mathbb{D}^i_{\epsilon,n}(P)$ is an open set.
\end{lemma}

\begin{proof}
Let $Q\in \mathbb{D}^i_{\epsilon,n}(P)$.  By definition, there are points $P',P''\in \mathcal{PS}_i(S)$ such that $P'$ and $P$ agree on $S_n$, $d_0(P',P'')<\epsilon$, and $Q\in A^i_n(P'')$.  Hence, $P'$ is an edge point with a represenation $(X',a,Y')$.  Since $d_0(P',P'')<\epsilon<\max({a,1-a})$, $P''$ is also an edge point on the same edge as $P'$: it has a represenation $P''=(X',a',Y')$ with $|a-a'|<\epsilon$.  Hence, $Q$ is also an edge point with a represenation $(X'',a',Y'')$ where $X''\in B^i_n(X')=B^i_n(X)$ and $Y''\in B^i_n(Y')=B^i_n(Y)$.
\par 
Now suppose $\{Q_j\}_{j\geq 1}\to Q$.  Then each $Q_j$ has a represenation $(X_j,a_j,Y_j)$ with $X_j\to X''$, $Y_j\to Y''$, and $a_j\to a'\in (a-\epsilon,a+\epsilon$.  Hence, for all sufficiently large $j$, we have $X_j\in B^i_n(X'')=B^i_n(X)$, $Y_j\in B^i_n$, and $a_j\in (a-\epsilon,a+\epsilon)$.  Hence $Q_j\in \mathbb{D}^i_{\epsilon,n}(P)$ for all sufficiently large $j$.
\end{proof}

\begin{lemma}
Let $X\in V_i(S)$, $n\in \mathbb{N}$, $0<\epsilon<1$, and let $Q\in \mathbb{D}^i_{\epsilon,n}(X)$ be an edge point.  Let $\{Q_j\}_{j\geq 1}\to Q$.  Then all but finitely many of the $Q_j$ are in $\mathbb{D}^i_{\epsilon,n}(X)$
\end{lemma}

\begin{proof}
Since $Q\in \mathbb{D}^i_{\epsilon,n}(X)$, there exists a vertex $X'\in B^i_n(X)$ and a point $P\in \mathcal{PS}_i(S)$ such that $d_0(X',P)<\epsilon$ and $Q\in A^i_n(P)$.  Since $d_0(X',P)<\epsilon<1$, $P$ is either an edge point on an edge incident to $X'$, or else $P=X'$.  We tackle these two cases seperately.

\par 
First, suppose $P=(X',a,Y)$ with $0<a<\epsilon$. Then $Q$ is also an edge point with a represenation $(X'',a,Y')$ where $X''\in B^i_n(X')=B^i_n(X)$ and $Y'\in B^i_n(Y)$.  Hence, for all sufficiently large $j$, $Q_j=(X_j,a_j,Y_j)$ with $X_j\in B^i_n(X'')=B^i_n(X)$, $Y_j\in  B^i_n(Y')=B^i_n(Y)$, and $0<a_j<\epsilon$.  Thus 
\[Q_j\in A^i_n((X',a_j,Y))\subseteq \mathbb{D}^i_{\epsilon,n}(X).\]

\par 
Next suppose $P=X'$.  Then $Q=(X'',a,Y)$ where $X'',Y\in B^i_n(X')=B^i_n(X)$.  Thus, for all sufficiently large $j$, $Q_j$ has a represenation $(X_j,a_j,Y_j)$ with $X_j,Y_j\in B^i_n(X')=B^i_n(X)$, hence $Q_j\in \mathbb{D}^i_{\epsilon,n}(X)$.
\end{proof}

\begin{lemma}
Let $X\in V_i(S)$, $n\in \mathbb{N}$, $0<\epsilon<1$, and $Q\in \mathbb{D}^i_{\epsilon,n}(X)$ is a vertex.  Let $\{Q_j\}_{j\geq 1}\to Q$.  Then all but finitely many of the $Q_j$ are in $\mathbb{D}^i_{\epsilon,n}(X)$
\end{lemma}

\begin{proof}
By definition, there exists points $P,P'\in \mathcal{PS}_i(S)$ such that $P\in B^i_n(X)$, $d_0(P,P')<\epsilon$ and $Q\in A^i_n(P')$.  Since $Q$ is a vertex, $P'$ must also be a vertex and $Q\in B^i_n(P')$.  Since $d_0(P,P')<1$, and both $P$ and $P'$ are vertices, $P=P'$, hence the vertices $X$, $P$, and $Q$ all agree on $S_n$.  
\par 
We must consider two subcases, depending on which of the conditions from Definition 6.1 the $Q_j$ satisfy.
\par 
In the first subcase, corresponding to Definition 6.1 (2), $Q_j=(X_j,a_j,Y_j)$ where $X_j\to Q$ and $a_j\to 0$.  For all  sufficiently large $j$, $X_j\in B^i_n(Q)=B^i_n(X)$ and $0\leq a_j<\epsilon$.  Thus $d_0(X_j,Q_j)=a_j<\epsilon$ and $X_j\in B^i_n(X)$, so $Q_j\in \mathbb{D}^i_{\epsilon,n}(X)$.

\par 
In the second subcase, corresponding to Definition 6.1 (3), $Q_j=(X_j,a_j,Y_j)$ with $X_j\to Q$ and $Y_j\to Q$. Thus, for all sufficiently large $j$, $X_j,Y_j\in B^i_n(Q)=B^i_n(X)$, so $Q_j\in A^i_n(X)\subseteq \mathbb{D}^i_{\epsilon,n}(X)$.
\end{proof}

\begin{corollary}
For all $P\in \mathcal{PS}_i(S)$, $\epsilon>0$, and $n\in \mathbb{N}$, the set $\mathbb{D}^i_{\epsilon,n}(P)$ is open.
\end{corollary}

\begin{proof}
Lemmas 6.14, 6.15, and 6.16 show that $\mathbb{D}^i_{\epsilon,n}(P)$ is open for all sufficiently small values of $\epsilon$.  If $Q\in \mathbb{D}^i_{\epsilon,n}(P)$, then there is a $0<\psi<\epsilon$ such that $\mathbb{D}^i_{\psi,n}(Q)\subset \mathbb{D}^i_{\epsilon,n}(P)$.  Hence $\mathbb{D}^i_{\epsilon,n}(P)$ is a union of open sets.
\end{proof}

\begin{lemma}
Let $U\subset \mathcal{PS}_i(S)$ be an open set, and let $P\in U$. Then there exists an $\epsilon>0$ such that $D^i_{\epsilon}(P)\subset U$.
\end{lemma}
\begin{proof}
First, assume $P$ is an edge point with a representation $P=(X,a,Y)$ (we will tackle the case where $P$ is a vertex seperately).  For sufficiently small $\epsilon>0$, note that $D^0_{\epsilon}(P)$ is contained entirely in the edge between $X$ and $Y$, and in particular it contains no vertices.  It follows from the definitions that $D^i_{\epsilon}(P)$ also contains no vertices.
\par
Consider the descending sequence of open sets $D^i_{\frac{1}{n}}(P)$ where $n\in \mathbb{N}$.  Suppose that for each $n$, there exists a point $P_n\in D^i_{\frac{1}{n}}(P)\backslash U$.  For all but finitely many $n$, $P_n$ is an edge point.  Hence, we can find representatives $P_n=(X_n,a_n,Y_n)$ where $X_n\in B_n^i(X)$, $Y_n\in B_n^i(Y)$, and $|a_n-a|<\frac{1}{n}$.  By construction, the sequence of points $\{P_n\}$ converges to $P$, and hence all but finitely many must be contained in $U$, which is a contradiction.  Thus, there is some $n$ such that $D^i_{\frac{1}{n}}(P)\subset U$.  
\par 
For $P$ a vertex, we can use a similar argument.   Consider the descending sequence of open sets $D^i_{\frac{1}{n}}(P)$ where $n\in \mathbb{N}$.  Suppose that for each $n$, there exists a point $P_n\in D^i_{\frac{1}{n}}(P)\backslash U$.  Each $P_n$ is either a vertex with $P_n\in B^i_n(P)$, or an edge point with a representation $P_n=(X_n,a_n,Y_n)$ such that $X_n\in B^i_n(P)$ and $0<a<\epsilon$.  But by definition, this is a sequence which converges to $P$, so all but finitely many of the $P_j$ are in $U$, a contradiction.
\end{proof}

\begin{lemma}
Let $U\subseteq \mathcal{PS}_i(S)$ be an open set and let $X\in U$ be a vertex.  Then there exists an $n\geq 0$ such that 
\[\mathbb{D}^i_{\frac{1}{n}}(X)\subseteq U\].
\end{lemma}

\begin{proof}
By Lemma 6.18, there exists an $m\geq 2$ such that $D^i_{\frac{1}{m}}(X)\subseteq U$.  Thus, it suffices to find an $n\geq 2$ such that 
\[\mathbb{D}^i_{\frac{1}{n}}(X)\subseteq D^i_{\frac{1}{m}}(X)\]
Suppose no such $n$ exists. Then for all $n\geq 2$, there exists a point $P_n$ such that 
\[P_n\in \left(\mathbb{D}^i_{\frac{1}{n}}(X)\right)\backslash\left(D^i_{\frac{1}{m}}(X)\right)\]
By definition, that inclusion means that there exists a vertex $X_n$ such that $X_n$ and $X$ agree on $S_n$ and $P_n\in D^i_{\frac{1}{n}}(X_n)$ but $P_n\notin D^i_{\frac{1}{m}}(X)$.
\par 
Suppose that $P_n$ is a vertex.  Then $P_n$ must agree with $X_n$ on $S_n$ but disagree with $X$ on $S_m$.  When $n\geq m$, this is a contradiction, since $X$ and $X_n$ agree on $S_n$.  
\par 
Hence, $P_n=(X'_n,a_n,Y'_n)$ is an edge point for all $n\geq m$.  But this implies that $a_n<\frac{1}{n}$ and that $X'_n$ agrees with $X$ on $S_n$.  Hence, the sequence $\{P_n\}_{n\geq m}$ converges to $X$, so all but finitely many of the $P_n$ are in $D^i_{\frac{1}{m}}(X)$, which is a contradiction.
\end{proof}

\begin{lemma}
Let $P$ and $Q$ be points in $\mathcal{PS}_i(S)$ that are either both vertices or both edge points.  Let $0\leq i\leq 4$ and $n\geq 0$ be integers and let $\epsilon>0$ be a real number. If $P$ and $Q$ $i$-agree on $S_n$, then 
\[\mathbb{D}^i_{\epsilon,n}(P)=\mathbb{D}^i_{\epsilon,n}(Q)\].
\end{lemma}

\begin{proof}
Immediate from the definitions.
\end{proof}

\begin{definition}\rm
An edge point $P\in \mathcal{PS}_i(S)$ is called \emph{rational} if it has a representation $P=(X,a,Y)$ where $a$ is rational.
\end{definition}

\begin{lemma}
Let $U\subset \mathcal{PS}_i(S)$ be an open set containing no vertices.  Then $U$ is the union of sets of the form $\mathbb{D}^i_{\epsilon,n}((X,a,Y))$ where $\epsilon$ is a rational number and $P$ is a rational edge point.
\end{lemma}
%%Can shorten proof by using Lemma 6.18 that the D^i_n generate all open sets
\begin{proof}
Let $P=(X,a,Y)\in U$.  We will try to find values of $n$ such that $\mathbb{D}^i_{\frac{1}{n}}(P)\subset U$.  
\par 
Suppose that no such $n$ exists.  Then for all $n\geq 0$, there exists a point 
\[P_n\in \left(\mathbb{D}^i_{\frac{1}{n}}(P)\backslash U\right)
\].  For all sufficiently large $n$, $X$ and $Y$ disagree on $S_n$.  Also, for all sufficiently large $n$, $0<a-\frac{1}{n}<a+\frac{1}{n}<1$.  Tracing through the definition of $\mathbb{D}^i_{\frac{1}{n}}(P)$, the statement statement $P_n\in \mathbb{D}^i_{\frac{1}{n}}(P)$ says that there are points $Q_n=(X'_n,a,Y'_n)$ and $R_n=(X'_n,a'_n,Y'_n)$ such that
\begin{enumerate}
\item
\[|a-a'_n|<\frac{1}{n}\].
\item
$P_n$ has a representation $(X''_n,a'_n,Y''_n)$.
\item
$X$, $X'_n$, and $X''_n$ all agree on $S_n$.
\item
$Y$, $Y'_n$, and $Y''_n$ all agree on $S_n$.
\end{enumerate}

In particular, the sequence of edge points $\{P_n\}$ have representations $\{(X''_n,a''_n,Y''_n)\}$ such that $X''_n\to X$, $Y''_n\to Y$, and $a''_n\to a$.  Thus $P_n\to P$, and hence all but finitely many of the $P_n$ are in $U$, a contradiction.  Thus, there is some $n$ such that $\mathbb{D}^i_{\frac{1}{n}}(P)\subset U$.
\par 
If $a$ is rational, then we are done.  If $a$ is irrational, then we can find a rational edge point $P'$ very close to $P$ such that $\mathbb{D}^i_{\frac{1}{n}}(P')\subset U$.
\end{proof}

We can now easily prove the Basis Theorem (Theorem 6.13).
\begin{proof}
Corollary 6.17 tells us that the $\mathbb{D}^i_{\epsilon,n}$ sets are all open, while Lemmas 6.19 and 6.22 says they generate all open sets.
\end{proof}

\begin{corollary}
$\mathcal{PS}_i(S)$ is a second-countable space.
\end{corollary}
\begin{proof}
Our goal is to find a countable subset of 
\[\{\mathbb{D}^i_{\epsilon}(P): \epsilon>0,P\in \mathcal{PS}_i(S)\}\]
which is still a basis.  For $m\in \mathbb{N}$, let $V_m$ be a set containing one vertex in each equivalence class for the equivalence relation "$i$-agrees on $S_m$."  By Lemma 4.15, $V_m$ is countable.  Likewise, let $E_m$ be the union of one edge in each equivalence class.  $E_m$ contains a countable number of edges (but an uncountable number of points.)  Let 
\[\tau=\bigcup_{m=0}^{\infty}(V_m\cup E_m)\subset \mathcal{PS}_i(S).\]
By Lemma 6.20, 
\[\{\mathbb{D}^i_{\epsilon}(P): \epsilon>0,P\in \mathcal{PS}_i(S)\}=\{\mathbb{D}^i_{\epsilon}(P): \epsilon>0,P\in \tau\}.\]

By Lemma 6.22, it suffices to consider basic open sets with rational $\epsilon$ values based around vertices and rational edge points.  Hence, if we let $\tau'$ be the subset of $\tau$ containing only vertices and rational edge points, then
\[\{\mathbb{D}^i_{\epsilon}(P): \epsilon\in \mathbb{Q}^+,P\in \tau'\}\]
is also a basis.  $\tau'$ is countable, hence the above basis is also countable.
\end{proof}

\section{Properties of $\mathcal{PS}_i(S)$}

In this section we study the basic properties of $\mathcal{PS}_i(S)$, as well as its homeomorphism group.  In particular, we show that it is metrizable, path-connected, not locally path-connected, and that every automorphism of $\mathcal{PS}_i(S)$ is isotopic to a unique graph automorphism of $\mathcal{P}(S)$.  We then show that the mapping class group of $\mathcal{PS}_i(S)$ is naturally isomorphic to $\Modpm(S)$, generalizing Margalit's result for finite-type pants graphs and confirming this case of the metaconjecture.

\begin{lemma}
The space $\mathcal{PS}_i(S)$ is first-countable.
\end{lemma}
\begin{proof}
By the Basis Theorem, each point $P\in \mathcal{PS}_i(S)$ has a countable neighborhood basis $\{\mathbb{D}^i_{\frac{1}{n}}(P):n\in \mathbb{N}\}$.  Thus $\mathcal{PS}_i(S)$ is first-countable.
\end{proof}

\begin{lemma}
The vertex set $V_i(S)$ is closed in $\mathcal{PS}_i(S)$.  
\end{lemma}

\begin{proof}
For an edge $(X,Y)$ in the pants graph, consider the $\frac{1}{2}$-ball around the midpoint:
\[\mathbb{D}^i_{\frac{1}{2}}((X,\frac{1}{2},Y))\]
\[=\bigcup_{(X',\frac{1}{2},Y')\mbox{ agrees with } (X,\frac{1}{2},Y)\mbox{ on }S_2} \quad \bigcup_{P':d_0(P',(X,\frac{1}{2},Y))<\frac{1}{2}} A^i_2(P')\]
\[=\bigcup_{(X',\frac{1}{2},Y')\mbox{ agrees with } (X,\frac{1}{2},Y)\mbox{ on }S_2} \quad \bigcup_{(X',a',Y'):0<a'<1} A^i_2((X',a',Y')\]

We note that this open set is a disjoint union of edges, and that it contains no vertices.  Moreover, every edge of $\mathcal{PS}_i(S)$ is contained in a similar neighborhood, so the union of all edges is open in $\mathcal{PS}_i(S)$, and hence $V_i(S)$ is closed.  
\end{proof}

\begin{lemma}
$\mathcal{PS}_i(S)$ is Hausdorff.  
\end{lemma}

\begin{proof}
First suppose $P_1=(X_1,a_1,Y_1)$ and $P_2=(X_2,a_2,Y_2)$ are distinct edge points, and that $a_1\neq a_2$.  Choose an $n\in \mathbb{N}$ such that $B_n^i(X_1)\cap B_n^i(Y_1)=B_n^i(X_2)\cap B_n^i(Y_2)=\emptyset$, and choose an $\epsilon>0$ such that $|a_1-a_2|<\frac{\epsilon}{2}$.  Then $\mathbb{D}^i_{\epsilon,n}(P_1)$ and $\mathbb{D}^i_{\epsilon,n}(P_2)$ are disjoint open neighborhoods.
\par 
On the other hand, if $a_1=a_2$, then $P_1$ and $P_2$ necessarily lie on distinct edges, since $P_1\neq P_2$.  Without loss of generality, we may assume that $X_1\neq X_2$.  Choose an $n\geq 1$ such that $B_n^i(X_1)\cap B_n^i(Y_1)=B_n^i(X_2)\cap B_n^i(Y_2)=B_n^i(X_1)\cap B_n^i(X_2)=\emptyset$.  Also choose an $\epsilon>0$ such that $\epsilon<\min(\{a,1-a\})$.  Then $\mathbb{D}^i_{\epsilon,n}(P_1)$ and $\mathbb{D}^i_{\epsilon,n}(P_2)$ are disjoint open neighborhoods.
\par 
A similar argument works when one or both points are vertices.
\end{proof}

The next two lemmas tell us the closure of the basis sets.  Intuitively, they say that we can get the closure of a basis set by replacing "$d_0(P',P'')<\epsilon$" in the definition of $D^i_{\epsilon,n}$ with  $d_0(P',P'')\leq \epsilon$.  They will be used to show that $\mathcal{PS}_i(S)$ is regular.

\begin{lemma}
Let $X\in \mathcal{PS}_i(S)$ be a vertex, and let $0<\epsilon\leq \frac{1}{2}$.  Then the closure of $\mathbb{D}^i_{\epsilon,n}(X)$ is 
\[
C^i_{\epsilon,n}=\bigcup_{X'\in B^i_n(X)}\quad \left(\bigcup_{P:d_0(P,X')\leq \epsilon} A^i_n(P)\right)
\]

\end{lemma}

\begin{proof}
It is clear from the definition of $C^i_{\epsilon,n}$ that each point in $C^i_{\epsilon,n}$ is a limit point of a sequence in $\mathbb{D}^i_{\epsilon,n}(X)$, and hence $C^i_{\epsilon,n}\subseteq \overline{\mathbb{D}^i_{\epsilon,n}(X)}$.  Our goal now is to show that $C^i_{\epsilon,n}$ is closed.  
\par 
Let $P_1, \ldots$ be a sequence in $\mathbb{D}^i_{\epsilon,n}$ which converges to some point $Q\in \mathcal{PS}_i(S)$.  We will show that $Q\in C^i_{\epsilon,n}$.
\par 
First suppose that $Q$ is an edge point.  Then $Q$ has a representation $Q=(Y,a,Z)$, $0<a<1$, and each $P_i$ has a represenations $P_j=(Y_j,a_j,Z_j)$ such that $\{Y_j\}\to Y$, $\{Z_j\}\to Z$, and $\{a_j\}\to a$.  Note that if $Y$ and $Z$ are both in $B^i_{n}(X)$, then $Q\in D^i_{\epsilon,n}(X)\subset\mathbb{D}^i_{\epsilon,n}(X)$ by definition and we are done.  On the other hand, At least one of $Y_j$ or $Z_j$ must be in $B^i_{n}(X)$, since $P_j\in \mathbb{D}^i_{\epsilon,n}(X)$.  Thus, we may assume that exactly one of $X_j$ or $Y_j$ is in $B^i_{n}(X)$.  
\par 
Due this assumption, the definition of $\mathbb{D}^i_{\epsilon}(X)$ implies that either $a_j<\epsilon\leq \frac{1}{2}$ or $a_j>1-\epsilon\geq \frac{1}{2}$.  Since the $a_j$ converge as a sequence of real numbers, then one of those two cases must occur only a finite number of times.  Assume, without loss of generality, that $a_j<\epsilon\leq \frac{1}{2}$ for all but finitely many $j\in \mathbb{N}$.  But that means $Y_j\in B^i_{n}(X)$ for all but finitely many $j\in \mathbb{N}$, and hence $Y\in B^i_{n}(X)$.  We also have $a\leq \epsilon$.  Thus $Q\in C^i_{\epsilon,n}$, as desired.

\par
We now turn our attention to the seperate case where $Q$ is a vertex.  As usual, there are two seperate subcases, depending on whether the $P_j$ satisfy conditions (2) or (3) of definition 6.1.  In the first subcase, we have represenations $P_j=(Y_j,a_j,Z_j)$ such that $Y_j\to Q$, $a_j\to 0$, and there are no restrictions on $Z_i$.  Since $(Y_j,a_j,Z_j)\in\mathbb{D}^i_{\epsilon,n}(X)$, either $Y_j\in B^i_n(X)$ and $a_j<\epsilon$, or $Z_j\in B^i_n(X)$ and $a_j>1-\epsilon$.  But the $a_j$ approach 0, so we must have $Y_j\in B^i_{n}(X)$ for all but finitely many $j\geq 1$.  Hence 
\[Q\in B^i_{n}(X)\subset \mathbb{D}^i_{\epsilon,n}(X)\subset C^i_{\epsilon,n}.\] 

\par 
In the second subcase, we have representations $P_j=(Y_j,a_j,Z_j)$ such that $Y_j\to Q$ and $Z_j\to Q$.  Hence, for large $j$, the vertices $Y_j$, $Z_j$, and $Q$ all agree on $S_n$.  Since $P_j\in \mathbb{D}^i_{\epsilon,n}(X)$ for all $j$, at least one of $Y_j$ or $Z_j$ agrees with $X$ on $S_n$.  Hence, by transitivity, for large $j$ the vertices $X$, $Y_j$, $Z_j$, and $Q$ all agree on $S_n$, thus
\[Q\in B^i_{n}(X)\subset \mathbb{D}^i_{\epsilon,n}(X)\subset C^i_{\epsilon,n}.\] 
\end{proof}

\begin{lemma}
Let $P=(X,a,Y)\in \mathcal{PS}_i(S)$ be an edge point, let $0<\epsilon<\max{a,1-a}$, and let $n\geq 1$ be large enough so that $X$ and $Y$ disagree on $S_n$. Then the closure of $\mathbb{D}^i_{\epsilon,n}(P)$ is 
\[
C^i_{\epsilon,n}=\bigcup_{P'=(X',a,Y'): X'\in B^i_n(X),Y'\in B^i_n(Y)}\quad \left( \bigcup_{P'':d_0(P',P'')\leq \epsilon} A^i_n(P'')\right)
\]
\[=\{(X',a',Y'):X'\in B^i_n(X),Y'\in B^i_n(Y),a-\epsilon\leq a'\leq a+\epsilon\}.\]

\end{lemma}

\begin{proof}
The equality in the lemma follows immediately from the definitions.  It is clear from the definition of $C^i_{\epsilon,n}$ that each point in $C^i_{\epsilon,n}$ is a limit point of a sequence in $D^i_{\epsilon,n}(P)$, and hence $C^i_{\epsilon,n}\subseteq \overline{D^i_{\epsilon,n}(P)}$.  Our goal now is to show that $C^i_{\epsilon,n}$ is closed.
\par 
Suppose $P_1, \ldots$ is a sequence of points in $D^i_{\epsilon,n}(P)$ which converge to some point $Q$. Note that $D^i_{\epsilon,n}(P)$ contains only edge points.  Since the $P_j$ converge, we can choose represenations $P_j=(W_j,a_j,Z_j)$ such that the $W_j$ converge as a sequence of points in $V_i(S)$ to some vertex $W$.  Each $W_j$ lies in either $B^i_n(X)$ or $B^i_n(Y)$.  Since we chose $n$ large enough such that $B^i_n(X)\cap B^i_n(Y)=\emptyset$, the only way for the $W_j$ to converge is if either $W_j\in B^i_n(X)$ for all but finitely many $j$, or $W_j\in B^i_n(Y)$ for all but finitely many $j$.  Without loss of generality, we assume that $W_j\in B^i_n(X)$ for \emph{all} $j$.  Hence, $W\in B^i_n(X)$, and $Z_j\in B^i_n(Y)$ for all $j$.  
\par 
Since $Z_j$ and $W$ disagree on $S_n$ for all $j$, the sequence $Z_1,\ldots$ does not converge to $W$.  Also, since $a-\epsilon\leq a_j\leq \epsilon$, the $a_j$ do not converge to either 0 or 1.  Hence, the only way for the $P_j$ to converge at all is if the $Z_j$ converge to a vertex $Z\in B^i_n(Y)$ and the $a_j$ converge to a point $b$ with $a-\epsilon\leq b\leq a+\epsilon$.  Thus $Q=(W,b,Z)\in C^i_{\epsilon,n}$.
\end{proof}

We recall that a space $\mathbb{X}$ is said to be $T_3$ if, for any closed set $C\subset \mathbb{X}$ and any point $x\in \mathbb{X}\backslash C$, there are disjoint open sets $U$ and $V$ with $x\in U$ and $C\subseteq V$.  A space is said to be \emph{regular} if it is both $T_3$ and Hausdorff.  
\par 
We also recall the well known theorem \cite[Chapter 4 Lemma 31.1  a]{munkres2000topology} that a Hausdorff space is regular if for every point $x$ and for every open set $x\in U$, there is an open set $V$ such that $x\in V$ and $\overline{V}\subset U$.

\begin{lemma}
$\mathcal{PS}_i(S)$ is a regular space.
\end{lemma}

\begin{proof}
Recall $\mathcal{PS}_i(S)$ is Hausdorff by Lemma 7.3. the previous two lemmas, each point $P\in \mathcal{PS}_i(S)$ has a neighborhood basis of open sets $\mathbb{D}^i_{\frac{1}{n}}(P)$ where $n\in \mathbb{N}$.  Moreover, for all sufficiently large $n$, we have $\overline{\mathbb{D}^i_{\frac{1}{n}}(P)}\subset \mathbb{D}^i_{\frac{1}{n-1}}(P)$, so $\mathcal{PS}_i(S)$ is regular.
\end{proof}

\begin{corollary}
$\mathcal{PS}_i(S)$ is metrizable.
\end{corollary}

\begin{proof}
The pants space is regular by the previous lemma, and second-countable by Corollary 6.23.  Thus, by the Urysohn Metrization Theorem\cite{willard2004general}, $\mathcal{PS}_i(S)$ is metrizable.
\end{proof}

\begin{lemma}
Let $\Gamma\subset \mathcal{P}(S)$ be a connected-component.  Then for $1\leq i\leq 4$, $\Gamma$ is dense in $\mathcal{PS}_i(S)$.
\end{lemma}

\begin{proof}
First, we note that by Lemma 4.23, every vertex in $\mathcal{PS}_i(S)$ is a limit point of $\Gamma$.  Now choose an edge point $P=(X,a,Y)$.  We want to show that $P\in \overline{\Gamma}$.   
\par 
Since $X$ and $Y$ differ by an elementary move, there exists curves $\alpha$ and $\beta$ such that $Y=(X\backslash\{\alpha\})\cup \{\beta\}$.  Let $S'$ be the complexity one subsurface in the complement of $X\backslash\{\alpha\}=Y\backslash\{\beta\}$.  Let $M$ be the least integer such that $S'\subset S_M$.  
\par 
For each $n\geq 1$, choose a vertex $X_n\in \Gamma$ such that $X_n$ and $X$ agree on $S_n$ (such a vertex must exist by Lemma 4.23).  For $n\geq M$, there is an elementary move on $X_n$ which replaces $\alpha$ by $\beta$.  Call the result of this elementary move $Y_n$: i.e. $Y_n=(X_n\backslash\{\alpha\})\cup \{\beta\}$.
\par 
Then there is an edge point $(X_n,a,Y_n)\in \Gamma$.  By construction, $\{X_n\}_{n\geq M}\to X$ and $\{Y_n\}_{n\geq M}\to Y$, thus $\{(X_n,a,Y_n)\}_{n\geq M}$ is a sequence of points in $\Gamma$ which converges to $P$.
\end{proof}

\begin{lemma}
For $1\leq i\leq 4$, $\mathcal{PS}_i(S)$ is a seperable space.
\end{lemma}

\begin{proof}
Let $\Gamma\subset \mathcal{P}(S)$ be a connected-component.  Then $\Gamma$ has countably many vertices and edges.  The subset of $\Gamma$ consisting of vertices and rational edge points is a countable dense subset of $\mathcal{PS}_i(S)$.
\end{proof}

\begin{lemma}
For $1\leq i\leq 4$, $\mathcal{PS}_i(S)$ is path-connected.
\end{lemma}

\begin{proof}
First, note that because $\mathcal{P}(S)=\mathcal{PS}_0(S)$ carries a finer topology than $\mathcal{PS}_i(S)$, every path-component of $\mathcal{P}(S)$ is also path-connected in $\mathcal{PS}_i(S)$.  Thus, it is sufficient to find a path between an arbitrary pair of vertices $X$ and $Y$ in different $\mathcal{P}(S)$-components.  
\par 
By Lemma 4.23, there exists a sequence of pants decompositions $X_1,X_2,\ldots$ such that $X=X_1$, $X_j$ and $X_{j+1}$ differ by an elementary move, and such that the sequence converges to $Y$ in $\mathcal{PS}_i(S)$.  Consider the function $f: [0,1]\to \mathcal{PS}_i(S)$ with 
\[
f(t)=
\begin{cases}
X_n &  t=1-1/n \\
Y & t=1 \\
(X_n,t-(1-1/n),X_n+1) & 1-1/n<t<1-1/(n+1)
\end{cases}
\]
We first claim that $f$ is continuous everywhere except possibly at $t=1$.  To see this, note that the restriction of $f$ to $[0,(1)$ is a continuous path in $\mathcal{P}(S)$, and $\mathcal{P}(S)$ has a coarser topology than $\mathcal{PS}_i(S)$.
\par 
Let $t_1,t_2,\ldots$ be a sequence of real numbers in $[0,1]$ which converges to 1.  Then for any $n\geq 1$, and for all but finitely many of the $t_j$, $f(t_j)$ lies on the edge between $X_M$ and $X_{M+1}$ for some $M\geq n$.  Thus, by the definition of pants convergence, $f(t_j)\to Y=f(1)$.  Thus $f$ is continuous at 1, and hence $X$ and $Y$ are in the same path-component of $\mathcal{PS}_i(S)$.
\end{proof}

\begin{lemma}
$\mathcal{PS}_i(S)$ is locally path-connected at a point $P$ if and only if $P$ is a vertex point.
\end{lemma}
\begin{proof}
First, suppose $P$ is a vertex.  Then $\mathbb{D}_{1/2}^i(X)$ is path-connected, by the same argument used in the previous lemma.
\par 
Now suppose $P=(X,a,Y)$ is an edge point.  Recall that $P$ has a neighborhood $E_i(S)=\mathcal{PS}_i(S)\backslash V_i(S)$ consisting of all edge points.  We will show that each edge is a connected component of $E_i(S)$.  
\par 
A single edge is path-connected and hence connected.  To show that $E_i(S)$ is disconnected, we consider the quotient space 
\[\hat{E}=E_i(S)/((X,a,Y)\sim (X,a',Y))\]
obtained by identifying all points on an edge.  We will show that $\hat{E}$ is totally disconnected.  Each point in $\hat{E}$ can be represented by an unordered pair $\{X,Y\}$ of adjacent vertices in $\mathcal{P}(S)$.  
\par 
For a point $\{X,Y\}\in \hat{E}$ and an integer $n\geq 1$, there is a neighborhood 
\[\hat{A}_n(\{X,Y\})=\{\{X',Y'\}:X'\in B^i_n(X)\mbox{ and } Y'\in B^i_n(Y)\}\]
This set is open because its preimage in $E_i(S)$ is open.  It follows from the disjointness lemma that for two points $\{X,Y\}$ and $\{X',Y'\}$, the neighborhoods $\hat{A}_n(\{X,Y\})$ and $\hat{A}_n(\{X',Y'\})$ are either disjoint or equal.  Thus, for sufficiently large $n$, the $\hat{A}_n$ partition $\hat{E}$ into disjoint clopen sets such that $\{X,Y\}$ and $\{X',Y'\}$ are in different parts of the partition.  Hence, $\hat{E}$ is totally disconnected.  Hence, each edge of $E_i(S)$ is a connected-component.
\par 
Note also that any smaller neighborhood $U$ such that $P\in U\subseteq E_i(S)$ must intersect infinitely many edges, and hence is also disconnected.  Thus, $\mathcal{PS}_i(S)$ is not locally connected at $P$.
\end{proof}

\begin{corollary}
$\mathcal{PS}_i(S)$ is not locally path-connected.  
\end{corollary}

\begin{proof}
Immediate from the previous lemma.
\end{proof}

\begin{corollary}
$\mathcal{PS}_i(S)$ is not homeomorphic to a CW complex.
\end{corollary}

\begin{proof}
This follows from the fact that $\mathcal{PS}_i(S)$ not locally path-connected.
\end{proof}
In particular, the preceeding corollary implies that $\mathcal{PS}_i(S)$ is different from the generalization of the pants graph in \cite{FossasParlier}.
%%Overhaul done through here

\begin{corollary}
Let $f: \mathcal{PS}_i(S)\to \mathcal{PS}_i(S)$ be a homeomorphism.  Then $f$ takes vertices to vertices and edge points to edge points.  
\end{corollary}

\begin{proof}
Homeomorphisms must preserve local-path-connectedness, so the result follows from the previous lemma.
\end{proof}

\begin{lemma}
Let $f: \mathcal{PS}_i(S)\to \mathcal{PS}_i(S)$ be a homeomorphism and let $P,Q\in \mathcal{PS}_i(S)$ be edge points.  Then $P$ and $Q$ lie in the same edge if and only if $f(P)$ and $f(Q)$ lie in the same edge.
\end{lemma}

\begin{proof}
By the previous corollary, $f$ restricts to a homeomorphism of $E_i(S)$.  Recall from the proof of Lemma 7.5 that $P$ and $Q$ lie on the same edge if and only if they are in the same connected-component of $E_i(S)$. 
\end{proof}

\begin{lemma}
Let $f: \mathcal{PS}_i(S)\to \mathcal{PS}_i(S)$ be a homeomorphism.  Let $X$ be a vertex and $P$ be an edge point.  Then $P$ lies on an edge incident to $X$ if and only if $f(P)$ lies on an edge incident to $X$.
\end{lemma}

\begin{proof}
Suppose $P$ is on an edge incident to $X$.  Then $P$ has a representative triple $P=(X,a,Y)$.  Consider the sequence of points $\{(X,\frac{a}{n},Y)\}_{n\geq 1}$.  This sequence converges to $X$ in $\mathcal{PS}_i(S)$.  Hence, the sequence $\{f((X,\frac{a}{n},Y))\}_{n\geq 1}$ converges to $f(X)$.  By the previous lemma, all the points in $\{f((X,\frac{a}{n},Y))\}_{n\geq 1}$ lie in a single edge, and hence their limit $f(X)$ must be one of the end points of that edge.  Thus $f(P)$ is on an edge incident to $f(X)$.  The converse follows from replacing $f$ by $f^{-1}$.
\end{proof}

\begin{theorem}
Let $S$ and $S'$ be infinite-type surfaces and let $f,f':\mathcal{PS}_i(S)\to \mathcal{PS}_i(S')$ be two homeomorphisms.  Then the following are equivalent:

\begin{enumerate}
\item
The homeomorphisms $f$ and $f'$ are isotopic.
\item
The restrictions of $f$ and $f'$ to $V_i(S)$ are equal.
\item 
For all edge points $P\in \mathcal{PS}_i(S)$, $f(P)$ and $f'(P)$ lie on the same edge in $\mathcal{PS}_i(S')$.
\end{enumerate}
\end{theorem}

\begin{proof}
\emph{(1) $\Rightarrow$ (2)}: Let $f$ and $f'$ be isotopic homeomorphisms.  By Corollary 7.14, $f$ and $f'$ restrict to homeomorphisms from $V_i(S)$ to $V_i(S')$, and any isotopy between $f$ and $f'$ restricts to an isotopy between $f_{|_{V_i(S)}}$ and $f'_{|_{V_i(S)}}$.  By Lemma 4.19, $V_i(S)$ is totally disconnected, so isotopic homeomorphisms on $V_i(S)$ must be equal.  Thus $f_{|_{V_i(S)}}=f'_{|_{V_i(S)}}$.
\par 
\emph{(2) $\Leftrightarrow$ (3)}: This implication is immediate from Lemma 7.16.
\par 
\emph{(2) and (3)$\Rightarrow$ (1)}:
Consider the edge between two vertices $X,Y\in V_i(S)$.  By hypothesis (2), $f(X)=f'(X)$ and $f(Y)=f'(Y)$.  By hypothesis (3), $f$ and $f'$ restrict to homeomorphisms from the edge between $X$ and $Y$ to the edge between $f(X)$ and $f(Y)$.  An edge in the pants space is homeomorphic to the real unit interval, which has trivial (orientation-preserving) mapping class group, hence the restrictions of $f$ and $f'$ to this edge are isotopic via the straight line isotopy.  We can apply the straight line isotopy to every edge simultaneously to get an isotopy between $f$ and $f'$.  
\end{proof}

\begin{definition}\rm
Let $S$ and $S'$ be infinite-type surfaces. A (not necessarily continuous) map $f:\mathcal{PS}_i(S)$ to $\mathcal{PS}_i(S')$ is a \emph{graph isomorphism} if 
\begin{enumerate}
\item 
The restriction of $f$ to $V_i(S)$ is a bijection between $V_i(S)$ and $V_i(S')$.
\item 
For all edge points $P=(X,a,Y)\in \mathcal{PS}_i(S)$, $f(P)=(f(X),a,f(Y))$.
\end{enumerate}
\end{definition}

\emph{Note}: In an earlier version of this paper, Theorem 7.18 was stated only for the special case $S=S'$.  The author is grateful to J. Aramayona for suggesting that theorem holds in the cases when $S \neq S'$.

\begin{theorem}
Let $S$ and $S'$ be infinite-type surfaces.  Then every homeomorphism from $\mathcal{PS}_i(S)$ to $\mathcal{PS}_i(S')$ is isotopic to a unique graph isomorphism from $\mathcal{P}(S)$ to $\mathcal{P}(S')$.
\end{theorem}

\begin{proof}
Let $f: \mathcal{PS}_i(S)\to \mathcal{PS}_i(S')$ be a homeomorphism.  By combining Corollary 7.14, Lemma 7.15, and Lemma 7.16, $f$ takes vertices to vertices, and for any edge point $P=(X,a_P,Y)$, $f(P)$ has a representative of the form $(f(X),a'_P,f(Y))$ for some $0<a'<1$.  By isotoping $f$ along the edges, we can make it so that $a_P=a'_P$ for all edge points.  Hence, $f$ is isotopic to a simplicial isomorphism.  
\par 
To show that this isomorphism is unique, suppose that $g,h:\mathcal{P}_i(S)\to \mathcal{P}_i(S')$ are two distict graph isomorphisms.  Then there exists some $X\in V_i(S)$ such that $g(X)\neq h(X)$.  Hence, by Theorem 7.17, $g$ and $h$ are not isotopic.
\end{proof}

Theorem 7.18 tells us that each isotopy class of homeomorphisms between pants spaces has a "canonical" representative: namely, the one which is a graph isomorphism of the underlying pants graphs.   In the particular case of $S=S'$, this means that the mapping class group of $\mathcal{PS}_i(S)$ is naturally isomorphic to the subgroup of $\Aut(\mathcal{P}(S))$ consisting of the automorphisms which are continuous in the topology of $\mathcal{PS}_i(S)$.  Our goal for the rest of this section is to determine which graph automorphisms of $\mathcal{P}(S)$ are homeomorphisms of $\mathcal{PS}_i(S)$.
\par 

\begin{lemma}
Let $\Gamma$ be a connected-component of $\mathcal{P}(S)$.  Let $f$ and $g$ be homeomorphisms of $\mathcal{PS}_i(S)$ which are also graph automorphisms.  Then $f=g$ if and only if $f\vert_{\Gamma}=g\vert_{\Gamma}$.
\end{lemma}
\begin{proof}
The result follows from the fact that $\Gamma$ is dense in $\mathcal{PS}_i(S)$.
\end{proof}

\begin{theorem}
Let $S$ and $S'$ be infinite-type surfaces.  A graph automorphism $\phi:\mathcal{PS}_i(S) \to \mathcal{PS}_i(S')$ is a homeomorphism of pants spaces if and only if it is induced by a homeomorphism from $S$ to $S'$.
\end{theorem}

\begin{proof}
The "if" direction follows immediately from the definition of the pants space topology.  Now suppose $\phi$ is a homeomorphism of pants spaces.  Choose a connected-component $\Gamma\in \Pi_0(\mathcal{P}(S))$.  Since $\phi$ is a graph isomorphism, the restriction of $\phi$ to $\Gamma$ is an isomorphism between $\Gamma$ and some $\Gamma'\in \Pi_0(\mathcal{P}(S'))$.  By Theorem 3.11, there is a homeomorphism $f:S\to S'$ which induces an isomorphism from $\Gamma$ to $\Gamma'$.  Then $f$ also induces a homeomorphism from $\mathcal{PS}_i(S)$ to $\mathcal{PS}_i(S')$, and by Lemma 7.20, this homeomorphism must be $\phi$. 
\end{proof}

This theorem immediately provides us with the following corollary.

\begin{corollary}
The mapping class groups $\Modpm(\mathcal{PS}_i(S))$ and $\Modpm(S)$ are naturally isomorphic.
\end{corollary}
We thus have a generalization to infinite-type surfaces of Margalit's result from \cite{margalit2002}

\section{A continuous group action on $\mathcal{PS}_i(S)$}
The goal of this section is to prove the following theorem.

\begin{theorem}
Let $S$ be an infinite-type surface with no planar ends.  Then the natural action of $\Modpm(S)$ on $\mathcal{PS}_i(S)$ is continuous.
\end{theorem}

First, we will need the following lemma.

\begin{lemma}
Let $S$ be an infinite-type surface.  Let $(f_1,X_1), (f_2,X_2),\ldots$ be a sequence in $\Modpm(S)\times V_i(S)$ which converges to $(f,X)$.  Then the sequence $f_1(X_1),f_2(X_2),\ldots$ converges to $f(X)$.
\end{lemma}

\begin{proof}
Fix some hyperbolic metric on $S$.  For each finite-type surface $S_n\subset S$, there is a compact subsurface $S'_n\subseteq S_n$ such that any geodesic simple closed curve is disjoint from $S_n\backslash S'_n$.  Since the $f_j$ converge to $f$ in the compact-open topoloy on $\Modpm(S)$, the restrictions of the $f_j$ to $S'_k$ uniformly converge to $f_{|S'_k}$ for all $k\geq 0$.  
\par 
Fix an $n\geq 1$.  Let $\hat{n}$ be the least natural number such that every curve in $f(X)$ which essentially intersects $S_n$ is contained in $S'_{\hat{n}}\subseteq S_{\hat{n}}$.  By uniform convergence, for all sufficiently large $m$, $f_{m_{|S'_{\hat{n}}}}=f_{|S'_{\hat{n}}}$, and $X_m$ agrees with $X$ on $S_{\hat{n}}$.  Thus, $f_m(X_m)$ agrees with $f(X)$ on $S_n$.  The result follows now from the definition of convergence in $V_i(S)$.
\end{proof}

Now we are ready to prove Theorem 9.1.
\begin{proof}
Suppose that $(f_1,P_1),\ldots$ is a sequence in $\Modpm(S)\times \mathcal{PS}_i(S)$ which converges to $(f,P)$.  We will show that $\lim_{n\to \infty} f_n(P_n)=f(P)$, and therefore conclude that the action of $\Modpm(S)$ is continuous.  As usual, we must consider three cases, depending on which of the conditions of definition 6.1 the $P_j$ satisfy.
\par 
First, we consider the case where $P=(X,a,Y)$ is an edge point.  Then there are represenations $P_j=(X_j,a_j,Y_j)$ such that $\{X_j\}\to X$, $\{Y_j\}\to Y$, and $a_j\to a$.  By the definition of the action of $\Modpm(S)$ on $\mathcal{PS}_i(S)$, $f_j(P_j)=(f(X_j),a_j,f(Y_j))$.  By Lemma 8.2, $f_j(X_j)\to f(X)$ and $f_j(Y_j)\to f(Y)$, fo $f_j(P_j)\to f(P)$, as desired.
\par 
Next, suppose that $P$ is a vertex, and that for $j\geq 1$, we have represenations $P_j=(X_j,a_j,Y_j)$ such that $\{X_j\}\to P$ and $\{a_j\}\to 0$.  Once again, we have $f_j(P_j)=(f_j(X_j),a_j,f(Y_j))$.  Again by Lemma 8.2, $f_j(X_j)\to f(X)$, and hence $f_j(P_j)\to f(P)$.
\par 
Finally, suppose $P$ is a vertex, and that there are represenations $P_j=(X_j,a_j,Y_j)$ where the sequences $\{X_j\}$ and $\{Y_j\}$ both converge to $P$.  Then $f_j(P_j)=(f_j(X_j),a_j,f(Y_j))$, so invoking Lemma 8.2 again shows that $\{f_j(X_j)\}$ and $\{f_j(Y_j)\}$ both converge to $f(P)$, and hence $f_j(P_j)\to f(P)$.
\end{proof}

\section{Open Questions}

There are two main open questions concerning $\mathcal{PS}_i(S)$ and $V_i(S)$ which we intend to tackle in a sequel paper.  We briefly discuss them in this section.

\subsection{Coarse Geometry of $V_i(S)$}

\begin{figure}
\def\svgwidth{\columnwidth}

\caption{Part of three pants decompositions of $S$.  The black circle is $\partial S_n$, the red dots are punctures, and the blue curves are pants curves.\label{fig:3pants}}
\subcaptionbox{X \label{X}}
{\includegraphics[scale=0.175]{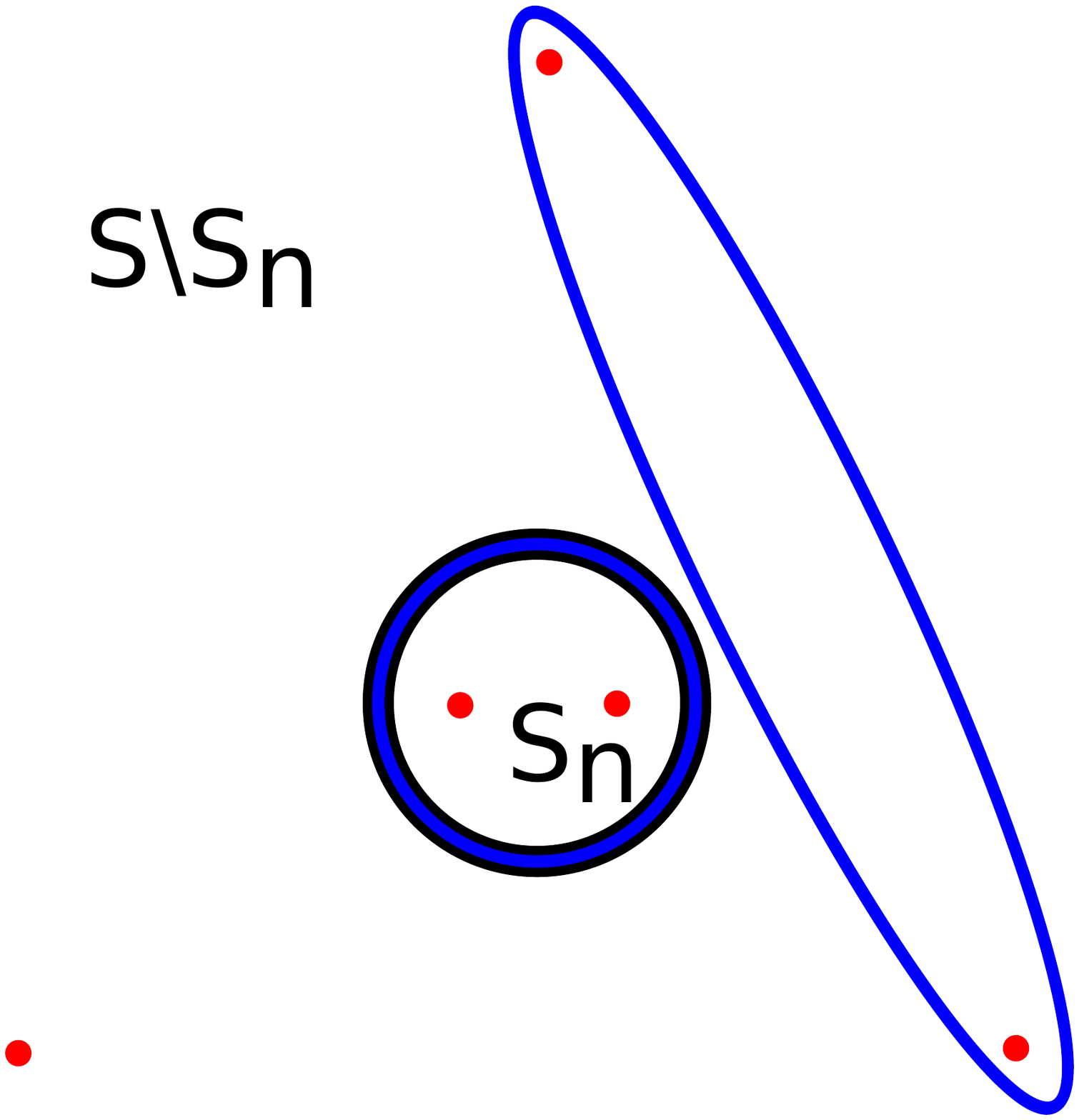}}
\subcaptionbox{Y\label{Y}}
{\includegraphics[scale=0.175]{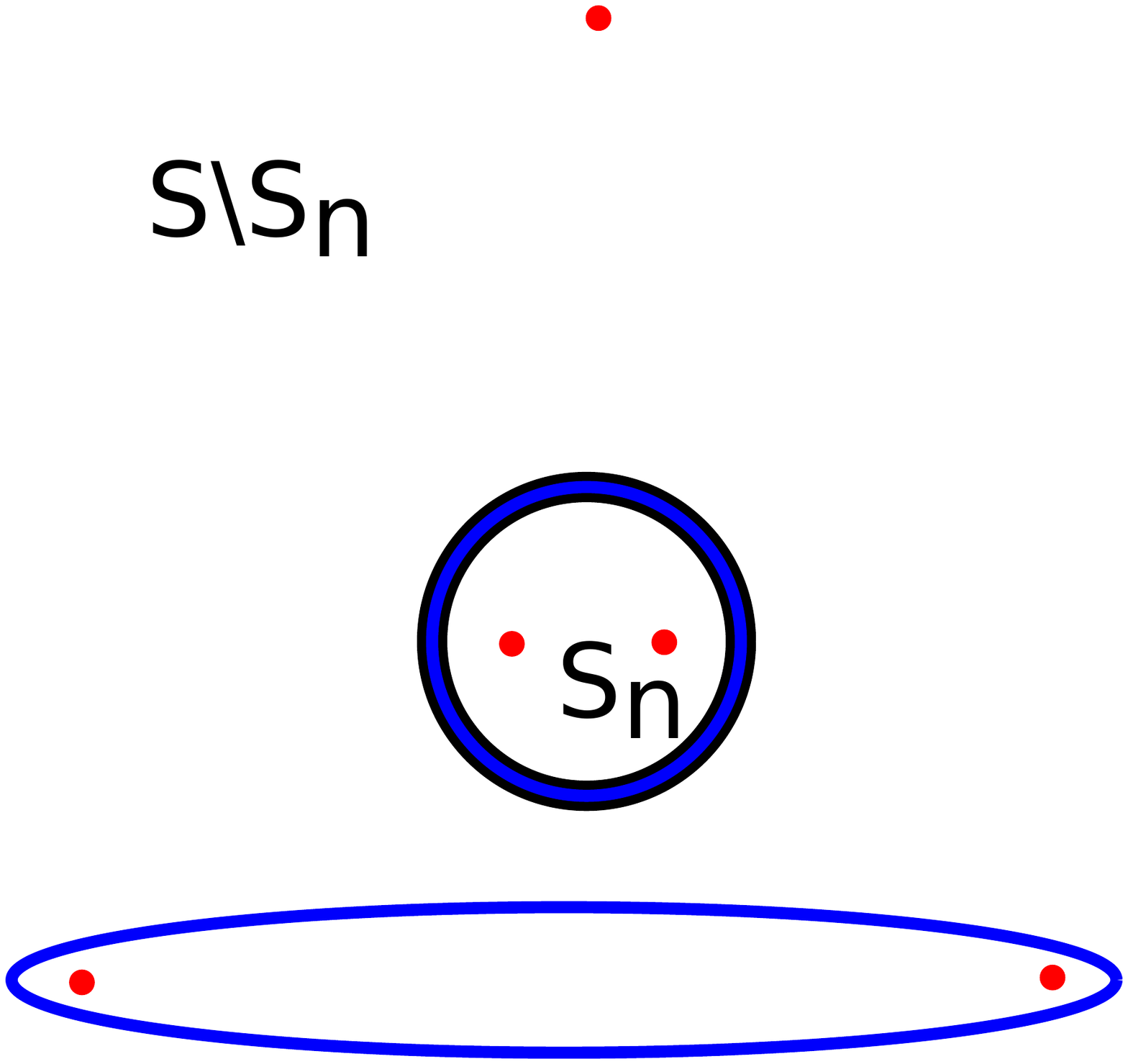}}
\subcaptionbox{Z\label{Z}}
{\includegraphics[scale=0.175]{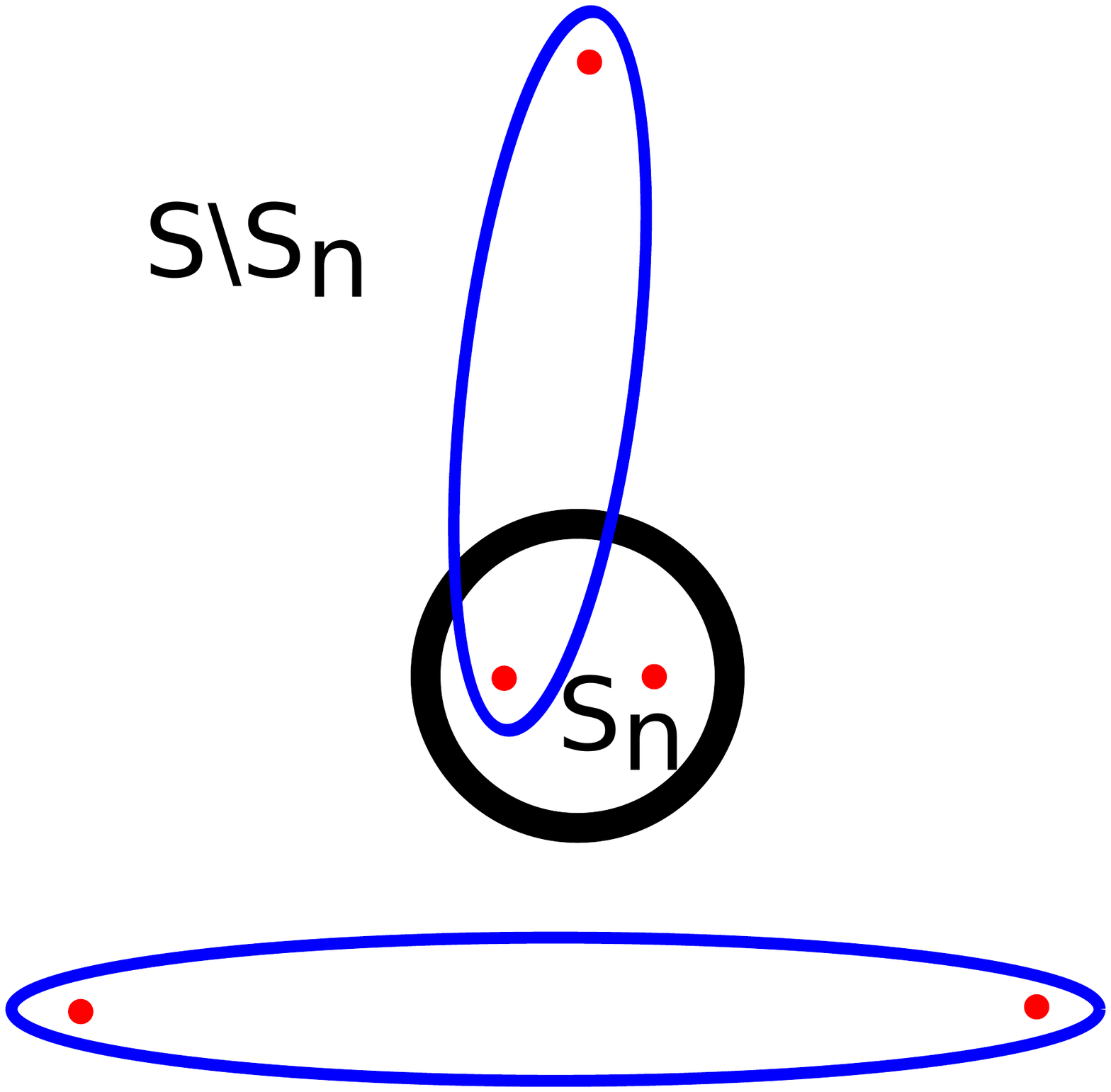}}

\end{figure}

In section 5, we constructed a metric on $V_i(S)$ and showed that it has diameter strictly greater than one. The definition of the metric allows us to "walk around" the vertex space using two kinds of steps: jumping to another pants decomposition which agrees on a finite-type subsurface, and making elementary moves.  Figure \ref{fig:3pants} establishes that these two types of moves do not commute.  Hence, it is nontrivial to compute distances greater than one in the vertex space.
\par 
Ultimately, we would like to know the diameter of $V_i(S)$.  If it is infinite diameter, we would like to understand its coarse geometry.
\subsection{Naturality of $V_i(S)$}
The topologies on $V_i(S)$ and $\mathcal{PS}_i(S)$ does not depend on the choice of the exhaustion of $S$.  However, the metric on $V_i(S)$ does depend on the exhaustion.  We do not know if there is some equivalence relation (such as quasi-isometry) under which the metric on $V_i(S)$ is independent of the exhaustion.

\subsection{Natural Metrics on $\mathcal{PS}_i(S)$}
In section seven, we showed that the pants space is metrizable using the Urysohn Metrization Theorem.  The proof of the Urysohn Metrization Theorem involves embedding any second-countable regular space in the Hilbert cube.  While this embedding gives a metric on $\mathcal{PS}_i(S)$, there is no guarentee that this metric has anything to do with the graph structure underlying the pants space, or with the family of metrics $d_i$ on the vertex space we constructed in section five.  
\par 
We would like to know the answer to the following questions.
\begin{question}\rm
Is there a metric on $\mathcal{PS}_i(S)$ such that the edges are geodesics and the length of an edge is the distance between its endpoints?
\end{question}

\begin{question}\rm
Is there a metric on the pants space such that the inclusion of $(V_i(S),d_i)\hookrightarrow \mathcal{PS}_i(S)$ is an isometric embedding?
\end{question}
\printbibliography
\end{document}